\documentclass[12pt,a4]{amsart}
\pdfoutput=1

\usepackage{a4wide}
\usepackage{slashbox}
\usepackage[all,cmtip]{xy}
\usepackage{amssymb}
\usepackage{enumitem}   
\usepackage{colortbl}
\usepackage{hyperref}
\usepackage[font=small,labelfont=bf]{caption} 
\usepackage{tkz-graph}
\tikzset{node distance=2cm, auto}
\usepackage{pdfpages}

\newtheorem{theorem}{Theorem}[section]
\newtheorem{corollary}[theorem]{Corollary}
\newtheorem{lemma}[theorem]{Lemma}
\newtheorem{proposition}[theorem]{Proposition}
 
\newtheorem{example}[theorem]{Example} 
\newtheorem{notation}[theorem]{Notation} 
\newtheorem{definition}[theorem]{Definition}
\newtheorem{note}[theorem]{Note}
\newtheorem{open}{Open Problem}

\newtheorem{remark}[theorem]{Remark}

\numberwithin{equation}{section}

\begin{document}
\title[Homogeneity of Inverse Semigroups	]%
{Homogeneity of Inverse Semigroups}

\author{Thomas Quinn-Gregson}
\email{tdqg500@york.ac.uk}
\address{Department of Mathematics\\University
  of York\\York YO10 5DD\\UK}





\thanks{This work forms part of my PhD at the University of York, supervised by Professor Victoria Gould, and funded by EPSRC}

\begin{abstract}
An inverse semigroup $S$ is a semigroup in which every element has a unique inverse in the sense of semigroup theory, that is, if $a \in S$ then there exists a unique $b\in S$ such that $a = aba$ and $b = bab$. We say that an inverse semigroup $S$ is a homogeneous (inverse) semigroup if any isomorphism between finitely generated (inverse) subsemigroups of $S$ extends to an automorphism of $S$. In this paper, we consider both these concepts of homogeneity for inverse semigroups, and show when they are equivalent. We also obtain certain classifications of homogeneous inverse semigroups, in particular periodic commutative inverse semigroups.  Our results may be seen as extending both the classification of homogeneous semilattices and the classification of certain classes of homogeneous groups, in particular the homogeneous abelian groups and homogeneous finite groups.  
\end{abstract}

\maketitle

\section{Introduction} \bigskip 
 
A \textit{structure} is a set $M$ together with a collection of finitary operations and relations defined on $M$. A countable structure $M$ is \textit{homogeneous} if any isomorphism between finitely generated (f.g.) substructures extends to an automorphism of $M$. There are strong links between homogeneity and model theoretic concepts such as $\aleph_0$-categoricity and quantifier elimination (see, for example, \cite[Theorem 6.4.1]{Hodges97}). As a consequence, homogeneity has been studied by several authors, and complete classifications have been obtained for a number of structures including graphs \cite{Lachlan80}, partially ordered sets \cite{Schmerl79} and semilattices \cite{Truss99}. There has also been much progress in the classification of homogeneous groups and rings (see, for example, \cite{Cherlin93}, \cite{Saracino84}) and the classification has been completed for finite groups in \cite{Cherlin2000} and \cite{Li99} and solvable groups in \cite{Cherlin91}. Homogeneous idempotent semigroups (bands)  have been classified by the author in \cite{Quinn}. In \cite{Dolinka}, a generalization of Hall's group for inverse semigroups is constructed. However, it is worth noting that the inverse semigroup obtained in \cite{Dolinka} is not a homogeneous inverse semigroup. 

 An inverse semigroup $S$ is a semigroup in which every element has a unique inverse in the sense of semigroup theory, that is, if $a \in S$ then there exists a unique $b \in S$ such that $a = aba$ and $b = bab$. We denote the inverse of $a$ as $a^{-1}$. It is clear that groups are inverse semigroups, as indeed are semilattices with binary operation of meet. The aim of this paper is to consider the homogeneity of inverse semigroups. Since an inverse semigroup can be viewed as either a semigroup or as a unary semigroup (a semigroup equipped with a basic unary operation), we have two concepts of homogeneity for inverse semigroups: as  semigroups and as inverse semigroups. 

This paper proceeds as follows: in Section 2 the theory of inverse semigroups required for this paper is given, and used in Section 3 to consider the substructure of a homogeneous inverse semigroup (HIS). In particular, we show that a HIS is either Clifford or bisimple. In Section 4 the homogeneity of Clifford semigroups is considered and, in the case where every element has finite order, is shown to depend only on Clifford semigroups with either surjective or trivial connecting morphisms. These results are then applied to inverse semigroups with finite maximal subgroups and, in Section 5, to commutative inverse semigroups, where a partial classification is obtained. Finally, in Section 6, we consider non-periodic inverse semigroups that are homogeneous as semigroups, showing that they are necessarily groups. In particular we find exactly when the two concepts of homogeneity for inverse semigroups intersect.   
All structures will be assumed to be countable.
 
\section{Basics of inverse semigroups} 

In this section we give a brief outline of the required theory of inverse semigroups (see \cite[Chapter 5]{Howie94} for a more complete overview). Throughout this section, we let $S$  denote an inverse semigroup. 

\begin{lemma}\cite[Proposition 5.1.3 (1)]{Howie94} \label{inverse rules} For any $a,b\in S$ we have $(a^{-1})^{-1}=a$ and $(ab)^{-1}=b^{-1}a^{-1}$.
\end{lemma} 

The following equivalence relations on $S$, known as the \textit{Green's relations}, are fundamental to the study of semigroups:
\[ a \, \mathcal{R} \, b \Leftrightarrow aa^{-1}=bb^{-1} \quad a\, \mathcal{L}\, b  \Leftrightarrow a^{-1}a=b^{-1}b, \quad a\, \mathcal{D} \,b \Leftrightarrow c\in S  \text{ such that  } a\, \mathcal{R}\, c\, \mathcal{L}\, b.
\]
Furthermore, we let $\mathcal{H}=\mathcal{R}\cap \mathcal{L}$. For $a\in S$ we let $R_a$ denote the $\mathcal{R}$-class containing $a$, and similarly for $L_a, H_a, D_a$. An element $e\in S$ is an \textit{idempotent} if $e^2=e$, and we denote the set of idempotents of $S$ by $E(S)$. If $e\in E(S)$ then $H_e$ is the maximal subgroup of $S$ with identity $e$  \cite[Corollary 2.2.6]{Howie94}. A commutative semigroup of idempotents is called an (algebraic) \textit{semilattice}. Every semilattice $Y$ comes equipped with a partial order, called the \textit{natural order on} $Y$, by $e\leq f$ if and only if $ef=e$.  
 A \textit{lower semilattice} is a poset in which the meet (denoted $\wedge$) of any pair of elements exists. The close link between algebraic semilattices and lower semilattices is highlighted in the following result. 
 
 \begin{proposition} Let $Y$ be an algebraic semilattice. Then $(Y,\leq)$ is a lower semilattice, where $\leq$ is the natural order on $Y$ and the meet of $a$ and $b$ in $Y$ is their product $ab$. Conversely, suppose $(Y,\leq)$ is a lower semilattice. Then $Y$ is an algebraic semilattice under the operation of meet, and $a\leq b$ if and only if $a\wedge b=a$ for each $a,b\in Y$.  
 \end{proposition} 
 
If $e,f\in E(S)$ with $e\not\geq f$ and $f\not\geq e$ then we say that \textit{$e$ and $f$ are incomparable}, and denote this by $e\perp f$. 
 The set of idempotents of any inverse semigroup forms a semilattice \cite[Theorem 5.1.1]{Howie94}. We may define a partial order on $S$, called the \textit{natural order on} $S$, given by $a\leq b$ if there exists $e\in E(S)$ such that $a=eb$. Note that $\leq$ is compatible with the multiplication $S$, that is, if $a\leq b$ and $c\in S$ then $ac\leq bc$ and $ca\leq cb$. Moreover, $\leq$ restricts to the natural order on $E(S)$. 
   
 Given a subset $A$ of $S$, we denote $\langle A \rangle$ and $\langle A \rangle_I$ as the subsemigroup and inverse subsemigroup, respectively, generated by $A$. If $A=\{a_1,\dots , a_n\}$, then it follows from Lemma \ref{inverse rules} that 
\[ \langle A \rangle_I= \langle a_1,\dots,a_n,a^{-1}_1,\dots,a_n^{-1} \rangle. 
\]
Hence all f.g. inverse semigroups are also f.g. semigroups, and thus every inverse  homogeneous semigroup is a HIS (since we are required to consider isomorphisms between all f.g. subsemigroups,  not only the inverse ones). 

Given an inverse subsemigroup $T$ of $S$, we call $T$ \textit{$n$-generated} ($n\in \mathbb{N}$) if there exists $t_1,\dots,t_n\in S$ such that $T=\langle t_1,\dots,t_n \rangle_I$. 

The \textit{order} of an element $a$ of $S$, denoted $o(a)$, is the cardinality of the monogenic inverse subsemigroup  $\langle a \rangle_I$. If all elements of $S$ are of finite order then $S$ is called \textit{periodic}, otherwise we call $S$ \textit{non-periodic}.

A semigroup is \textit{bisimple} if it has a single $\mathcal{D}$-class. A fundamental example of a bisimple inverse semigroup is the \textit{bicylic monoid}, the inverse monoid presented as an inverse monoid by $\langle x:xx^{-1}>x^{-1}x\rangle$, with identity element $xx^{-1}$ \cite{Clif&Pres61}. In particular, bicyclic monoids are non-periodic inverse semigroups with chain of idempotents $xx^{-1}>x^{-1}x>x^2x^{-2}>x^3x^{-3}>\cdots$. 
 
 The inverse semigroup $S$ is \textit{completely semisimple} if no distinct $\mathcal{D}$-related idempotents are related under the natural order on $E(S)$. This is equivalent to $S$ not containing a copy of the bicyclic monoid. Indeed, if $e,f\in E(S)$ is such that $e>f$ and $e \, \mathcal{D} \, f$ then there exists $x\in S$ with $xx^{-1}=e$ and $x^{-1}x=f$, and it follows that $\langle x \rangle_I$ is isomorphic to the bicyclic monoid (see \cite{Goberstein06}, for example). The converse is immediate.

We say that $S$ is \textit{Clifford} if $xx^{-1}=x^{-1} x$ for all $x\in S$ or, equivalently by \cite[Theorem 4.2.1]{Howie94}, if each $\mathcal{D}$-class is a group. In this case the idempotents of $S$ are central and $S/\mathcal{D}$ forms a semilattice.


\section{Properties of Homogeneity} 

Our methods for proving homogeneity comes in two forms: either we prove it directly from certain isomorphism theorems or we use the general method of Fra\"iss\'e. We shall now outline the later method. Here we apply this only to inverse semigroups, and for the general case we refer to \cite[Chapter 7]{Hodges93}. 

Let $\mathcal{K}$ be a class of f.g. inverse semigroups. Then we say   
\begin{enumerate} [label=(\arabic*)] 
\item $\mathcal{K}$ is \textit{countable} if  it contains only countably many isomorphism types.
\item  $\mathcal{K} $ is \textit{closed under isomorphism} if whenever $A\in \mathcal{K}$ and $B\cong A$ then $B\in \mathcal{K}$.  
\item $\mathcal{K}$ has the \textit{hereditary property} (HP) if given $A\in \mathcal{K}$ and $B$ a f.g. inverse subsemigroup of $A$ then $B\in \mathcal{K}$. 
\item  $\mathcal{K}$ has the \textit{joint embedding property} (JEP) if given $B_1,B_2\in \mathcal{K}$, then there exists $C\in \mathcal{K}$ and embeddings $f_i:B_i\rightarrow C$ ($i=1,2$).
\item  $\mathcal{K}$ has the \textit{amalgamation property}\footnote{This is also known as the \textit{weak amalgamation property}.}  (AP) if given  $A, B_1, B_2\in \mathcal{K}$, where $A$ is non-empty, and embeddings $f_i:A\rightarrow B_i$ ($i=1,2$), then there exists $D\in \mathcal{K}$ and embeddings $g_i: B_i \rightarrow D$  such that 
\[     f_1 \circ g_1 = f_2 \circ g_2. 
\] 
\end{enumerate} 

The \textit{age} of an inverse semigroup $S$ is the class $\mathcal{K}$ of all f.g. inverse semigroups which can be embedded in $S$. We may now apply Fra\"iss\'e's Theorem \cite{Fraisse} to the case of inverse semigroups. 

\begin{theorem}[Fra\"iss\'e's Theorem for inverse semigroups] Let  $\mathcal{K}$ be a non-empty countable class of f.g. inverse semigroups which  is closed under isomorphism and satisfies HP, JEP and AP. Then there exists a unique, up to isomorphism, countable HIS $S$ such that $\mathcal{K}$ is the age of $S$. Conversely, the age of a countable HIS is closed under isomorphism, is countable and satisfies HP, JEP and AP. 
\end{theorem} 


We call $S$ the \textit{Fra\"iss\'e limit} of $\mathcal{K}$. 

\begin{example} {\em Let $\mathcal{K}$ be a  Fra\"iss\'e class of commutative inverse semigroups. Then the Fra\"iss\'e limit $S$ of $\mathcal{K}$ is also commutative inverse, for if $a,b\in S$ then $\langle a,b\rangle_I\in \mathcal{K}$, and so $ab=ba$ (this may be generalised to arbitrary varieties of inverse semigroups).  }
\end{example} 

\begin{example}{\em The class of all finite semilattices forms a Fra\"iss\'e class, with Fra\"iss\'e limit the \textit{universal semilattice} (see \cite{Hall75}, for example). }
\end{example}

\begin{example} {\em Let $\mathcal{K}$ be the class of all f.g. Clifford semigroups. Then $\mathcal{K}$ is closed under both substructure and (finite) direct product, and thus has JEP. However it was shown in \cite{Hall75} that AP does not hold. Similarly, the class of all f.g. inverse semigroups is not a Fra\"iss\'e class.  }
\end{example}

Fra\"iss\'e's Theorem will be of particular use when we consider the homogeneity of commutative inverse semigroups in Section 5. We now consider the substructure of a HIS.

Given a structure $M$, we  call a subset $N$ \textit{characteristic} it is invariant under each automorphism of $M$, that is, $N\phi=N$ for each $\phi \in \text{Aut}(M)$. 

 It is easily shown that every characteristic subset generates a characteristic substructure. Moreover, the homogeneity of a structure will pass to characteristic substructures (the result for groups is given in \cite[Lemma 1]{Cherlin91}). However we shall view a larger class of substructures:
 
 \begin{definition}{\em Let $M$ be a structure with substructure $A$ such that if $\phi$ is an automorphism of $M$ and there exists  $a,b \in A$ with $a\phi = b$, then $\phi|_A$ is an automorphism of $A$. Then we call $A$ a \textit{quasi-characteristic substructure of $M$}. }
 \end{definition} 

\begin{lemma}\label{rel-char equiv} Let $A$ be a substructure of $M$. Then the following are equivalent: 
\begin{enumerate}[label=(\roman*), font=\normalfont]
\item  $A$ is a quasi-characteristic substructure of $M$; 
 \item if $\phi\in \text{Aut}(M)$ is such that there exists $a,b \in A$  with $a\phi = b$, then $x\phi\in A$ for all $x\in A$.
  \end{enumerate}
\end{lemma} 

\begin{proof} (i) $\Rightarrow$ (ii). Trivial. 

(ii) $ \Rightarrow $ (i). Let $\phi\in \text{Aut}(M)$ be such that there exists $a,b\in A$ with $a\phi =b$. Then by our hypothesis $\phi|_A$ is a map from $A$ to $A$ and, being a restriction of an automorphism, is an injective morphism. Moreover, as $b\phi^{-1}=a$ and $\phi^{-1}\in \text{Aut}(M)$, we have that $x\phi^{-1}\in A$ for all $x\in A$. Hence $\phi|_{A}$ is surjective, and thus an automorphism.  
\end{proof}   

\begin{example} \label{Y char}{\em Let $S$ be an inverse semigroup with semilattice of idempotents $Y$. Then for any $\phi \in \text{Aut}(S)$ and $e\in Y$ we have 
\[ (e\phi)^2=e^2\phi=e\phi
\]
and it follows that $Y$ is a characteristic subsemigroup, and thus quasi-characteristic. }
\end{example}

\begin{example} \label{G rel char}{\em Let $S$ be an inverse semigroup and $e$ an idempotent of $S$. Then the maximal subgroup $H_e$ of $S$ is quasi-characteristic. Indeed, note that $\mathcal{H}$ (and indeed each Green's relation) is preserved by morphisms of $S$, since if $\theta:S\rightarrow T$ is a morphism to an inverse semigroup $T$ then, for $a,b\in S$,
\begin{align*}
a \, \mathcal{H} \, b & \Rightarrow aa^{-1}=bb^{-1} \text{ and } a^{-1}a=b^{-1}b, \\
& \Rightarrow (aa^{-1})\theta= (bb^{-1})\theta \text{ and } (a^{-1}a)\theta = (b^{-1}b)\theta, \\
& \Rightarrow (a\theta)(a\theta)^{-1} = (b\theta)(b\theta)^{-1}  \text{ and } (a\theta)^{-1}(a\theta)=(b\theta)^{-1}(b\theta) \\
& \Rightarrow a\theta \, \mathcal{H} \, b\theta,
\end{align*}
since $(a\theta)^{-1}=a^{-1}\theta$. Hence, if  $\phi$ is an automorphism of $S$ and $a\phi=b$ for some $a,b\in H_e$, then for any $x\in H_e$ we have
\[ x \, \mathcal{H} \, a \Rightarrow x\phi \, \mathcal{H} \, a\phi=b
\] 
and so $x\phi \in H_e$ as required. }
\end{example} 

\begin{lemma}\label{rel char homog} Let $M$ be a homogeneous structure with a quasi-characteristic substructure $A$. Then $A$ is homogeneous. 
\end{lemma} 

\begin{proof} Let $\phi$ be an isomorphism between f.g. substructures $N$ and $N'$ of $A$. Then as $N$ and $N'$ are  f.g. substructures of $M$, we may extend $\phi$ to $\bar{\phi}\in \text{Aut}(M)$. Since  $N\bar{\phi}=N'$ and $A$ is quasi-characteristic we have $\bar{\phi}|_A\in \text{Aut}(A)$, and so $A$ is homogeneous. 
\end{proof} 
 
 Given a subset $T$ of $S$, we say that Aut($S$) \textit{acts transitively} on $T$ if for any $a,b\in T$, there exists an automorphism $\phi$ of $S$ such that $a\phi=b$. 
 
\begin{lemma}\label{trans on Y} If $S$ is a HIS then $\text{Aut}(S)$ acts transitively on $E(S)$. 
\end{lemma} 

\begin{proof} Given $e,f\in E(S)$, we have $\langle e \rangle_I=\{e\}\cong \{f\}=\langle f \rangle_I$, and so the result follows by the homogeneity of $S$. 
\end{proof} 

\begin{corollary}\label{iso max subgrps} The maximal subgroups of a HIS are isomorphic. 
\end{corollary} 

\begin{proof} Let $H_e$ and $H_f$ be maximal subgroups of a HIS $S$ for some idempotents $e,f$. Then by Lemma \ref{trans on Y} there exists an automorphism $\phi$ of $S$ such that $e\theta=f$. Hence $H_e \phi = H_f$ since $\mathcal{H}$ is preserved by automorphisms. 
\end{proof} 
 
We now consider the property of homogeneity as an inverse semigroup for two key classes of inverse semigroups: groups and semilattices. Note that, for a group, the two notions of an inverse of an element coincide. Moreover, since a finitely generated inverse subsemigroup of a group will contain a unique idempotent, it will be a subgroup \cite[page 62]{Howie94}. Hence a group is a HIS if and only if it is a homogeneous group. 

 Homogeneity of lower semilattices was first considered by Droste in \cite{Droste85} and, together with Truss and Kuske, in \cite{Truss99}. Note that both articles consider the homogeneity of the corresponding structure $(Y,\leq,\wedge)$. Since any morphism of the corresponding (algebraic) semilattice $(Y,\wedge)$ preserves $\leq$, their work effectively considers homogeneity of (algebraic) semilattices.
Moreover, for each $e\in Y$, as $e^{-1}=e$ we have $\langle e_1,\dots,e_n \rangle_I=\langle e_1,\dots,e_n \rangle$ for any $e_1,\dots,e_n\in Y$. Hence a semilattice is a HIS if and only if it is a homogeneous semigroup.

 We shall therefore refer simply to {\em homogeneous semilattices} without ambiguity. Note however, that a homogeneous semilattice need not be homogeneous as a poset.

Every linearly ordered set is a semilattice, where the meet function $\wedge$ is the minimum of two elements. The unique, up to isomorphism, non-trivial homogeneous linear order is $(\mathbb{Q},\leq)$ (where $\leq$ is the natural order on $\mathbb{Q}$). Moreover, $(\mathbb{Q},\leq)$ is the unique homogeneous semilattice also being homogeneous as a poset. 

A semilattice $Y$ with associated partial order $\leq$ is called a \textit{semilinear order} if $(Y,\leq)$ is non-linear and, for all $\alpha\in Y$, the set $\{\beta\in Y:\beta\leq \alpha\}$ is linearly ordered. This is equivalent to $Y$ not containing a \textit{diamond}, that is, distinct $\delta,\alpha,\gamma,\beta\in Y$ such that $\delta>\{\alpha,\gamma\}>\beta$ and $\alpha \perp \gamma$ with $\alpha\gamma=\beta$.  In \cite{Droste85}, every homogeneous semilattice forming a semilinear order was constructed, which lead to the following classification. 

\begin{proposition}[{{\cite{Droste85, Truss99}}}]  A non-trivial homogeneous semilattice is isomorphic to either $(\mathbb{Q},\leq)$, the universal semilattice, or a semilinear order. 
\end{proposition} 

 Note that not every semilinear order is a homogeneous semilattice. Moreover, a non-trivial homogeneous semilattice is \textit{dense}, that is, for any $\alpha>\beta$, there exists $\delta$ such that $\alpha>\delta>\beta$. We will use this property of homogeneous semilattices throughout the paper without reference. The universal semilattice is the unique homogeneous semilattice which contains a diamond, and we have the following consequence of its universal property. 

\begin{lemma}\label{U fact} Let $Y$ be the universal semilattice and $Z$ a finite subsemilattice of $Y$. Then for any finite semilattice $Z'$ in which $Z$ embeds, there exists $X\subset Y\setminus Z$ such that $Z\cup X\cong Z'$. 
\end{lemma}  

\begin{proof} Let $Z'$ be a finite semilattice and $\theta:Z\rightarrow Z'$ an embedding.  Since $Y$ is universal, there exists an embedding $\phi:Z'\rightarrow Y$. Hence $\theta\phi$ is an isomorphism between $Z$ and $Z\theta\phi$, which we can extend to an automorphism $\chi$ of $Y$, since the universal semilattice is homogeneous. Then $Z$ is a subsemilattice of $(Z'\phi)\chi^{-1}$ since 
\[ Z\chi = Z\theta\phi \subseteq Z'\phi. 
\] 
Moreover, $(Z'\phi)\chi^{-1}$ is isomorphic to $Z'$, and the result follows by taking $X=(Z'\phi)\chi^{-1} \setminus Z$.
\end{proof}

We summarise our findings on an arbitrary HIS.  
 
\begin{corollary}\label{substructure} Let $S$ be a HIS with semilattice of idempotents $Y$. Then $Y$ is a homogeneous semilattice and the maximal subgroups of $S$ are pairwise isomorphic homogeneous groups. 
\end{corollary} 

 Since a homogeneous finite semilattice is trivial, and an inverse semigroup with a unique idempotent is a group we obtain: 

\begin{corollary} Let $S$ be a finite inverse semigroup. Then $S$ is a HIS if and only if it is a homogeneous group. 
\end{corollary} 

\begin{theorem}\label{cliff of bisimple} Let $S$ be a HIS. If $S$ is completely semisimple then it is Clifford, otherwise $S$ is bisimple. 
\end{theorem} 

\begin{proof} Let $S$ be completely semisimple HIS. Suppose for contradiction that there exist distinct $\mathcal{D}$-related idempotents $e,f$, so that $e \, \perp \, f$. Since $\mathcal{D}$ is preserved by automorphisms of $S$, it follows by Lemma \ref{trans on Y} that each $\mathcal{D}$-class contains the same number of idempotents. Hence there exists $g\in E(S)$ with $g \, \mathcal{D} \, ef$ and $g \neq ef$. Then $e>ef$ as $e \, \perp f$, and so $D_e\neq D_g$ by the semisimplicity of $S$.  

We claim that $e>g$. Indeed, if $g>e$ then $g>ef$, contradicting $S$ being completely semisimple. If $e \perp g$ then then there exists an isomorphism between $\langle e,f \rangle_I = \{e,f,ef\}$ and $\langle e,g \rangle_I=\{e,g,eg\}$, which fixes $e$ and sends $f$ to $g$. Extending to an automorphism $\phi$ of $S$, we have
\[  D_{e}\phi=D_{e} \neq D_g = D_{f}\phi, 
\] 
contradicting $D_e=D_f$, and the claim follows. Similarly, $f>g$, so that $e,f>ef \geq g$ as $\leq$ is compatible with multiplication, and so $ef=g$, a contradiction. Hence $e=f$ and $S$ is Clifford. 

Suppose instead that $S$ is not completely semisimple, so that there exists $\mathcal{D}$-related idempotents $e',f'$ with $e'>f'$. Let $h,k\in E(S)$. If $h>k$ or $k>h$ then $\{h,k\}\cong \{e',f'\}$ and so $h \, \mathcal{D} \, k$ by homogeneity. On the other hand, if $h\perp k$ then $\{h,hk\}\cong \{e',f'\} \cong \{k,hk\}$ yields $h \, \mathcal{D} \, hk \, \mathcal{D} \, k$. Thus $S$ is bisimple.  
\end{proof}

\begin{proposition}\label{bisimple inf} Let $S$ be a bisimple HIS. Then each maximal subgroup of $S$ is infinite and $\mathcal{H}$ is not a congruence. 
\end{proposition} 

\begin{proof} Since $S$ is bisimple, there exists $x\in S$ with $\langle x \rangle_I$ isomorphic to the bicyclic monoid, with chain of idempotents $xx^{-1}>x^{-1}x> x^{-2}x^2>x^{-3}x^3>\cdots$. 
For each $n>2$, let $\theta_n$ be an automorphism of $S$ extending the unique isomorphism between the chain of idempotents $\{xx^{-1},x^{-1}x,x^{-2}x^2\}$ and $\{xx^{-1},x^{-1}x,x^{-n}x^n\}$ (such a list of automorphisms exist by the homogeneity of $S$). Let $x\theta_n=y_n$.  Then $(xx^{-1})\theta_n=y_n y_n^{-1}=xx^{-1}$ and similarly $y_n^{-1}y_n=x^{-1}x$, so that $x \, \mathcal{H} \, y_n$ for each $n$. Furthermore, 
\begin{equation} \label{bic iso} (x^{-2}x^2)\theta_n=y_n^{-2}y_n^2=x^{-n}x^n,
\end{equation} 
 so that if $y_n=y_m$ then $x^{-n}x^n=x^{-m}x^m$, and so $n=m$. Hence $\{y_n:n>2\}$ is an infinite subset of $H_x$, and thus each $\mathcal{H}$-class (and in particular each maximal subgroup) is infinite by \cite[Lemma 2.2.3]{Howie94}. 

Suppose for contradiction that $\mathcal{H}$ is a congruence on $S$. Then as $x \, \mathcal{H} \, y_3$ we have  $x^2 \, \mathcal{H} \, y_3^2$, and so by \eqref{bic iso}
 \[ x^{-3}x^3=y_3^{-2}y_3^2=x^{-2}x^2,
 \] 
a contradiction.
\end{proof}

\begin{open} Is a bisimple HIS a group? 
\end{open}

\section{The Clifford case}

In this section we consider the homogeneity of Clifford semigroups. The following construction is taken from \cite[Chapter 4]{Howie94}. 

Let $Y$ be a semilattice. To each $\alpha \in Y$ associate a group $G_{\alpha}$ with identity element $e_{\alpha}$, and assume that $G_{\alpha} \cap G_{\beta} = \emptyset$ if $\alpha \neq \beta$. For each pair $\alpha, \beta \in Y$ with $\alpha \geq \beta$, let $\psi_{\alpha, \beta}: G_{\alpha} \rightarrow G_{\beta}$ be a homomorphism, and assume that the following conditions hold: 
\begin{align} & \text{for all } \alpha \in Y,  \text{ } \psi_{\alpha, \alpha} = 1_{G_{\alpha}}, \textit{the identity automorphism of } G_{\alpha}, \\
& \text{for all } \alpha, \beta, \gamma\in Y \text{ such that } \alpha \geq \beta \geq \gamma, 
\end{align} 
\[  \psi_{\alpha, \beta} \psi_{\beta, \gamma} = \psi_{\alpha, \gamma}.  \] 
On the set $S=\bigcup_{\alpha \in Y} G_{\alpha}$ define a multiplication by 
\[ a * b = (a \psi_{\alpha, \alpha \beta})(b \psi_{\beta, \alpha \beta})
\] 
for $a\in G_{\alpha}, b \in G_{\beta}$, and denote the resulting structure by  $S=[Y;G_{\alpha}; \psi_{\alpha, \beta}]$. Then $S$ is a semigroup, and is called a \textit{strong semilattice $Y$ of groups $G_{\alpha}$} ($\alpha\in Y$).  

We follow the usual convention of denoting element an element $a$ of $G_{\alpha}$ as $a_{\alpha}$. By \cite[Lemma IV.1.7]{Petrich99}, the natural order $\leq$ on $S$ is equivalent to 
\[ a_{\alpha} \geq b_{\beta} \text{ if and only if } \alpha \geq \beta \text{ and } a_{\alpha}\psi_{\alpha,\beta}=b_{\beta} 
\] 
for each $a_{\alpha},b_{\beta}\in S$. It follows immediately that the natural order on $S$ restricts to equality on maximal subgroups. 

\begin{theorem}\cite[Theorem 4.2.1]{Howie94} A semigroup is Clifford if and only if it is isomorphic to a strong semilattice of groups. 
\end{theorem} 

 To understand homogeneity of Clifford semigroups, we require a better understanding of their f.g. inverse subsemigroups. Since the class of all Clifford semigroups forms a variety of inverse semigroups \cite{Howie94}, the inverse subsemigroups are also Clifford. 

\begin{lemma} \label{sub clifford} Let $T$ be an inverse subsemigroup of a Clifford semigroup $S=[Y;G_{\alpha};\psi_{\alpha,\beta}]$. Then there exists a subsemilattice $Y'$ of $Y$, and subgroups $H_{\alpha}$ of $G_{\alpha}$ for each $\alpha\in Y'$ such that 
\[ T=[Y';H_{\alpha};\psi_{\alpha,\beta}|_{H_{\alpha}}].
\]
\end{lemma} 

\begin{proof} Since $T$ is inverse, there exists a subsemilattice $Y'$ of $Y$ such that $E(T)=\{e_{\alpha}:\alpha\in Y'\}$. If $g_{\alpha},h_{\alpha}\in T$ then $h_{\alpha}^{-1}\in T$ since $T$ is inverse, so $g_{\alpha}h_{\alpha}^{-1}\in T$. It follows that the maximal subgroup of $T$ containing $e_{\alpha}$ is a subgroup $H_{\alpha}$ of $G_{\alpha}$. Moreover, if $\alpha>\beta$ in $Y'$ then  $e_{\beta}\in T$ and so if $g_{\alpha}\in T$ then 
\[ g_{\alpha}e_{\beta}=(g_{\alpha}\psi_{\alpha,\beta})(e_{\beta}\psi_{\beta,\beta})=(g_{\alpha}\psi_{\alpha,\beta})(e_{\beta})=g_{\alpha}\psi_{\alpha,\beta}\in T,
\]
 and so $H_{\alpha}\psi_{\alpha,\beta}\subseteq H_{\beta}$. Hence the homomorphism $\psi_{\alpha,\beta}|_{H_{\alpha}}:H_{\alpha}\rightarrow H_{\beta} $ is well defined, and the result follows. 
\end{proof} 

The following isomorphism theorem for Clifford semigroups is a simple consequence of \cite[Lemma IV.1.8]{Petrich99}.

\begin{theorem}\label{iso clifford} Let $S=[Y; G_{\alpha}; \psi_{\alpha, \beta}]$ and $T=[Z; H_{\gamma}; \varphi_{\gamma, \delta}]$ be a pair of Clifford semigroups. Let $\pi: Y\rightarrow Z$ be an isomorphism and, for each $\alpha \in Y$, let $\theta_{\alpha}: G_{\alpha}\rightarrow H_{\alpha \pi}$ be an isomorphism. Assume also that for any $\alpha \geq \beta$, the diagram  
\begin{align} \label{1} \xymatrix{
G_{\alpha} \ar[d]^{\psi_{\alpha, \beta}} \ar[r]^{\theta_{\alpha}} &H_{\alpha \pi} \ar[d]^{\varphi_{\alpha \pi, \beta \pi}} \\\
G_{\beta} \ar[r]^{\theta_{\beta}} &H_{\beta \pi}}
\end{align}  
 commutes. Define a map $\theta$ on $S$ by 
 \begin{align*} 
  \theta: a \mapsto a \theta_{\alpha} \text{    if } a\in G_\alpha.  
 \end{align*} 
 Then $\theta$ is an isomorphism from $S$ into $T$, denoted $\theta=[\theta_{\alpha},\pi]_{\alpha\in Y}$. Conversely, every isomorphism from $S$ to $T$ can be so constructed.  
\end{theorem}  

We denote the diagram (\ref{1}) by $[\alpha,\beta;\alpha\pi,\beta\pi]$, or  $[\alpha,\beta;\alpha\pi,\beta\pi]^S$ if the semigroup $S$ needs highlighting. The isomorphism $\pi$ is called the \textit{induced isomorphism} from $Y$ to $Z$. Note that the diagram $[\alpha,\alpha;\alpha\pi,\alpha\pi]$ commutes for any isomorphism $\pi:Y\rightarrow Z$ and isomorphisms $\theta_{\alpha}:G_{\alpha}\rightarrow H_{\alpha\pi}$ ($\alpha\in Y)$ since 
\[ \psi_{\alpha,\alpha}\theta_{\alpha}=1_{G_{\alpha}}\theta_{\alpha} =\theta_{\alpha}= \theta_{\alpha}1_{H_{\alpha\pi}}=\theta_{\alpha}\varphi_{\alpha\pi,\alpha\pi}. 
\] 
Consequently, we need only check that the diagram $ [\alpha,\beta;\alpha\pi,\beta\pi]$  commute for  each $\alpha>\beta$. 

\begin{remark}\label{remark 1} {\em If $S=[Y;G_{\alpha};\psi_{\alpha,\beta}]$ and $H_{\alpha}\cong G_{\alpha}$ for each $\alpha\in Y$ then  the isomorphism theorem above can be used to construct a Clifford semigroup isomorphic to $S$ with maximal subgroups $H_{\alpha}$. Formally, if $\theta_{\alpha}:G_{\alpha}\rightarrow H_{\alpha}$ is an isomorphism for each $\alpha\in Y$ then $\theta=[\theta_{\alpha},1_Y]_{\alpha\in Y}$ is an isomorphism from $S$ to $T=[Y;H_{\alpha};\varphi_{\alpha,\beta}]$, where 
\[ \varphi_{\alpha,\beta}=\theta_{\alpha}^{-1}\psi_{\alpha,\beta} \theta_{\beta}. 
\] 
In particular, maximal subgroups which can be written as a direct sum (d.s.) or a direct product (d.p.) of groups, can be considered as an internal or external d.s./d.p. without problems arising.

 We adopt a non standard notation by denoting the internal d.s. and internal d.p. of a pair of groups $H$ and $H'$ as $ H \oplus H'$ and $H \otimes H'$, respectively. We also denote the internal direct sum of $n$ copies of a group $H$ by $H^{n}$ , where $n\in \mathbb{N}^*=\mathbb{N}\cup \{\aleph_0 \}$. Unless stated otherwise, we assume that all d.s.'s of groups are internal. }
\end{remark} 

If $S=[Y;G_{\alpha};\psi_{\alpha,\beta}]$ is a HIS, then as the groups $G_{\alpha}$ are the maximal subgroups of $S$ and  $Y\cong E(S)$, Corollary  \ref{substructure} may be rewritten as: 

\begin{corollary}\label{structure clifford} If $S=[Y;G_{\alpha};\psi_{\alpha,\beta}]$ is a HIS then $Y$ is homogeneous and the groups $G_{\alpha}$ are pairwise isomorphic homogeneous groups. 
\end{corollary} 

Hence, if $S=[Y;G_{\alpha};\psi_{\alpha,\beta}]$ is a HIS and $G_{\alpha}\cong G$ then, by Remark \ref{remark 1}, we may consider each group $G_{\alpha}$ to be a labelling of $G$, and the morphisms $\psi_{\alpha,\beta}$ to be a labelling of an endomorphism of $G$. 

\begin{definition} {\em  Let $S=[Y;G_{\alpha};\psi_{\alpha,\beta}]$ be a Clifford semigroup. We may consider a stronger form of homogeneity by saying that $S$ is a \textit{structure-HIS} if, given any pair of f.g. inverse subsemigroups $A=[Z;H_{\gamma};\psi_{\gamma,\delta}^H]$ and $A'=[Z';H_{\gamma'}';\psi_{\gamma',\delta'}^{H'}]$ of $S$ and any isomorphism $\theta=[\theta_{\alpha},\pi]_{\alpha\in Z}$ from $A$ to $A'$, then for any automorphism $\hat{\pi}$ extending $\pi$, there exists an automorphism $\hat{\theta}=[\hat{\theta}_{\alpha},\hat{\pi}]_{\alpha\in Y}$ of $S$ extending $\theta$.} 
\end{definition} 

Note that if $S$ is a structure-HIS then for every automorphism $\pi$ of $Y$, there exists an automorphism of $S$ with induced automorphism $\pi$. Indeed, for any $\alpha\in Y$, the isomorphism $\phi$  between the trivial inverse subsemigroups $\{e_{\alpha}\}$ and $\{e_{\alpha\pi}\}$ has induced isomorphism $\pi|_{\{\alpha\}}$ (from $\{\alpha\}$ to $\{\alpha\pi\}$). The result is then immediate by the structure-homogeneity of $S$. 

This form of homogeneity will be particularly useful for the construction of a Clifford  HIS from particular Clifford semigroups. 
We first consider how the characteristic subgroups of the maximal subgroups of a Clifford semigroup give rise to characteristic inverse subsemigroups. 
 
A subset $T$ of a Clifford semigroup $S$ will be called \textit{order-characteristic} if whenever $T$ contains an element of order $n$, then every element of order $n$ in $S$ belongs to $T$.

Given a group $G$ with subset $A$, then we set 
\[ o(A):=\{n: \text{there exists } a\in A \text{ such that } o(a)=n\}. 
\]
Note that if $a_{\alpha} \in S$ then, as $a_{\alpha}$ is contained in the group $G_{\alpha}$, the inverse subsemigroup $\langle a_{\alpha} \rangle_I$ is a cyclic group. In particular, our definition of the order of an element intersects with the group theory definition, that is $o(a_{\alpha})$ is the minimal $n>1$ such that $a_{\alpha}^n=e_{\alpha}$. Hence, as cyclic groups of the same cardinality are isomorphic, the following generalization of \cite[Lemma 1]{Cherlin2000} and its corollary are easily verifiable: 

\begin{lemma} \label{order els} Let $S$ a Clifford HIS with characteristic subset $T$. Then $T$ is order-characteristic. 
\end{lemma} 

\begin{corollary} \label{order els2} Let $S$ and $S'$ be a pair of isomorphic Clifford HIS's with characteristic inverse subsemigroups $T$ and $T'$, respectively such that $o(T)=o(T')$. Then $T\cong T'$, and if $S=S'$ then $T=T'$.  
\end{corollary} 

\begin{lemma}\label{char sub his} Let $S=[Y;G_{\alpha};\psi_{\alpha,\beta}]$ be a Clifford semigroup. For each $\alpha\in Y$, let $H_{\alpha}$ be an order-characteristic subgroup of $G_{\alpha}$ such that $H_{\alpha}\cong H_{\beta}$ for all $\alpha,\beta \in Y$. Then 
\[ T=[Y;H_{\alpha};\psi_{\alpha,\beta}|_{H_{\alpha}}] 
\] 
is an order-characteristic inverse subsemigroup of $S$. In particular, if $S$ is a HIS then so is $T$. 
\end{lemma} 

\begin{proof}  Notice that as each $H_{\alpha}$ are isomorphic order-characteristic subgroups, and as $o(H_{\alpha}\psi_{\alpha,\beta})\subseteq o(H_{\beta})$, it follows that  $H_{\alpha}\psi_{\alpha,\beta}\subseteq H_{\beta}$ for all $ \alpha\geq \beta$ in $Y$, and so $T$ is well defined. The result is then immediate. 
%
\end{proof}

In particular, if $S$ in Lemma \ref{char sub his} is a HIS, then the result holds if $H_{\alpha}$ is characteristic by Lemma \ref{order els}.

 A pair of subsets $A$ and $B$ of a group $G$ are of \textit{coprime order} if $o(A)\cap o(B)\subseteq \{1\}$. If $G$ is periodic then this is equivalent to being of \textit{relatively prime exponent}, defined in \cite{Cherlin91}, but does not require the theory of supernatural numbers. Note that if $G=A\otimes B$ where $A$ and $B$ are periodic of coprime order, then clearly $A$ and $B$ are order-characteristic subgroups of $G$, and so the lemma above may be used in this case.

 \begin{corollary} \label{coprime char subgroups} Let $G$ be a homogeneous group with characteristic subsets $H$ and $K$ such that $H\cap K\subseteq\{1\}$. Then $H$ and $K$ are of coprime order. 
\end{corollary} 

\begin{proof}
If $h\in H$ and $k\in K$ both have order $n\in \mathbb{N}^*$, then by Lemma \ref{order els} $H$ and $K$ both contain all elements of order $n$. Since $H$ and $K$ intersect trivially, it follows that $n=1$, and so the subgroups are coprime. 
\end{proof}
 
The subsequent pair of lemmas will be vital to our work, and proofs will be omitted. 

\begin{lemma} \label{extend isos} Let $G=H\otimes K$ be a group with $H$ and $K$ periodic of coprime order. Then, for each subgroup $G'$ of $G$, there exists subgroups $H'$ and $K'$ of $H$ and $K$, respectively, such that $G'=H'\otimes K'$. 
\end{lemma} 
 
\begin{lemma}\label{iso of direct prod} Let $G_1=H_1\otimes K_1$ and $G_2=H_2\otimes K_2$ be a pair of groups with the $H_i$ and $K_i$ periodic of coprime orders for each $i=1,2$, and $H_1\cong H_2$, $K_1\cong K_2$. Let $G_1'=H_1'\otimes K_1'$ and $G_2'=H_2'\otimes K_2'$ be subgroups of $G_1$ and $G_2$, respectively, and $\theta^H:H_1'\rightarrow H_2'$ and $\theta^K:K_1'\rightarrow K_2'$ be a pair of morphisms. Then the map $\theta$ given by
\[ (hk)\theta = (h\theta^H)(k\theta^K) \quad (h\in H_1',k\in K_2')
\] 
is a morphism from $G_1'$ to $G_2'$, and every morphism can be so constructed. 
\end{lemma} 

The homomorphism $\theta$ in the lemma above will often be denoted as $\theta^H\otimes \theta^K$. We observe that Lemmas \ref{extend isos} and \ref{iso of direct prod} fail in general if we drop the periodic condition. 

If $G=H \otimes K$ is a group with $H$ and $K$ periodic of coprime order then clearly $H$ and $K$ are characteristic subgroups. The following simplification of \cite[Lemma 1]{Cherlin91} then follows from the pair of lemmas above. 

\begin{corollary}\label{coprime homog}  Let $G=H \otimes K$ be a group with the $H$ and $K$ periodic of coprime order. Then $G$ is homogeneous if and only if $H$ and $K$ are homogeneous.
\end{corollary} 

Given a group $G=H\otimes K$ where $H$ and $K$ are periodic of coprime order, consider the Clifford semigroup $S=[Y;G_{\alpha};\psi_{\alpha,\beta}]$ with $G_{\alpha}\cong G$ for each $\alpha$. Then $G_{\alpha}=H_{\alpha}\otimes K_{\alpha}$ where $H_{\alpha}\cong H$ and $K_{\alpha}\cong K$, and by Lemma \ref{iso of direct prod} we may let $\psi_{\alpha,\beta}=\psi_{\alpha,\beta}^H \otimes \psi_{\alpha,\beta}^K$ where $\psi_{\alpha,\beta}^H: H_{\alpha} \rightarrow H_{\beta}$ and $\psi_{\alpha,\beta}^K: K_{\alpha} \rightarrow K_{\beta}$. Since $H$ and $K$ are order-characteristic subgroups of $G$, it follows that the sets 
\begin{equation*}
S^H:=[Y;H_{\alpha};\psi_{\alpha,\beta}^H]  \text{ and } S^K :=[Y;K_{\alpha};\psi_{\alpha,\beta}^K]
\end{equation*} 
are characteristic inverse subsemigroups of $S$ by Lemma \ref{char sub his}.

\begin{corollary}\label{auto of sum clifford} Let $S=[Y;H_{\alpha}\otimes K_{\alpha};\psi_{\alpha,\beta}]$ be a periodic Clifford semigroup, where each $H_{\alpha}$ and $K_{\alpha}$ are of coprime order. Let $\pi$ be an automorphism of $Y$ and $\theta^H=[\theta_{\alpha}^H,\pi]_{\alpha\in Y}$ and $\theta^K=[\theta_{\alpha}^K,\pi]_{\alpha\in Y}$ be automorphisms of $S^H$ and $S^K$, respectively. Letting $\theta_{\alpha}=\theta^H_{\alpha} \otimes \theta^K_{\alpha}$, then $\theta=[\theta_{\alpha},\pi]_{\alpha \in Y}$ is an automorphism of $S$, and all automorphisms of $S$ can be constructed in this way.  
\end{corollary} 

\begin{proof} We show first that $\theta$ is an automorphism of $S$. By Lemma \ref{iso of direct prod} each $\theta_{\alpha}$ is an isomorphism, so it remains to prove that the diagram $[\alpha,\beta;\alpha\pi,\beta\pi]$ commutes for any $\alpha> \beta$. Let $g_{\alpha}\in G_{\alpha}$, say, $g_{\alpha}=h_{\alpha}k_{\alpha}$ ($h_{\alpha}\in H_{\alpha}, k_{\alpha}\in K_{\alpha}$). Then 
\begin{align*}
 g_{\alpha}\theta_{\alpha}\psi_{\alpha\pi,\beta\pi} & =(h_{\alpha}\theta_{\alpha}^H\psi_{\alpha\pi,\beta\pi}^H)(k_{\alpha}\theta_{\alpha}^K \psi_{\alpha\pi, \beta\pi}^K) \\
& = (h_{\alpha}\psi_{\alpha,\beta}^{H}\theta_{\beta}^{H})(k_{\alpha}\psi_{\alpha,\beta}^{K}\theta_{\beta}^{K}) \\
& = g_{\alpha}\psi_{\alpha,\beta}\theta_{\beta} 
\end{align*}
since $[\alpha,\beta;\alpha\pi,\beta\pi]^{S^H}$ and $[\alpha,\beta;\alpha\pi,\beta\pi]^{S^K}$ commutes. Hence $[\alpha,\beta;\alpha\pi,\beta\pi]^S$ commutes and $\theta$ is an automorphism of $S$. 
The converse follows from Theorem \ref{iso clifford} and the fact that $S^H$ and $S^K$ are characteristic inverse subsemigroups of $S$. 
\end{proof}

\begin{proposition}\label{structure sums} Let $S=[Y;H_{\alpha}\otimes K_{\alpha};\psi_{\alpha,\beta}]$ be a periodic Clifford semigroup, where each $H_{\alpha}$ and $K_{\alpha}$ are of coprime order. Then $S$ is a structure-HIS if and only if $S^H$ and $S^K$ are structure-HIS.  
\end{proposition} 

\begin{proof} ($\Rightarrow$) Follows immediately as $S^H$ and $S^K$ are characteristic and have structure semilattice $Y$. 

($\Leftarrow$) Let $A$ and $B$ be a pair of f.g. inverse subsemigroups of $S$. It follows from Lemmas \ref{sub clifford},  \ref{extend isos} and \ref{iso of direct prod} that we may take 
\begin{align*}
& A=[Z';H_{\gamma}'\otimes K_{\gamma}'; \psi_{\gamma,\delta}^{H'}\otimes \psi_{\gamma,\delta}^{K'}] \\
& B= [Z'';H_{\tau}''\otimes K_{\tau}''; \psi_{\tau,\sigma}^{H''}\otimes \psi_{\tau,\sigma}^{K''}]
\end{align*}
where $H_{\gamma}'$ and $K_{\gamma}'$ are subgroups of   $H_{\gamma}$ and $K_{\gamma}$, respectively,  $\psi_{\gamma,\delta}^{H'}=\psi_{\gamma,\delta}|_{H_{\gamma}'}$ and $\psi_{\gamma,\delta}^{K'}=\psi_{\gamma,\delta}|_{K_{\gamma}'}$. Similarly for $B$. 

 Let $\theta=[\theta_{\gamma},\pi]_{\gamma\in Z'}$ be an isomorphism from $A$ to $B$, and $\hat{\pi}$ an automorphism of $Y$ which extends $\pi$. Then $\theta_{\gamma}= \theta_{\gamma}^{H'} \otimes \theta_{\gamma}^{K'}$ for some isomorphisms $\theta_{\gamma}^{H'}:H_{\gamma}'\rightarrow H_{\gamma\pi}''$ and $\theta_{\gamma}^{K'}:K_{\gamma}'\rightarrow K_{\gamma\pi}''$. Hence $\theta^{H'}=[\theta_{\gamma}^{H'},\pi]_{\gamma\in Z'}$ is an isomorphism from $ [Z';H_{\gamma}';{\psi}_{\gamma,\delta}^{H'}]$ onto $[Z'';H_{\tau}'';{\psi}_{\tau,\sigma}^{H''}]$ and similarly for $\theta^{K'}=[\theta_{\gamma}^{K'},\pi]_{\gamma\in Z'}$. 
  
Since $\theta^{H'}$ is an isomorphism between f.g. inverse subsemigroups of the structure-HIS $S^H$, we may extend $\theta^{H'}$ to an automorphism $[\theta_{\alpha}^{H},\hat{\pi}]_{\alpha\in Y}$  of $S^H$, and similarly extend $\theta^{K'}$ to an automorphism $[\theta_{\alpha}^K,\hat{\pi}]_{\alpha\in Y}$ of $S^K$. Hence $[\theta_{\alpha}^H\otimes \theta_{\alpha}^K,\hat{\pi}]_{\alpha\in Y}$ is an automorphism of $S$ by Corollary \ref{auto of sum clifford}, and extends $\theta$ as required. 
\end{proof} 

An simple adaptation of the proof above also gives the following result. 

\begin{proposition}\label{structure sums 2} Let $S=[Y;H_{\alpha}\otimes K_{\alpha};\psi_{\alpha,\beta}]$ be a periodic Clifford semigroup, where $H_{\alpha}$ and $K_{\alpha}$ are of coprime order, and $S^H$ is a structure-HIS. Then $S$ is a HIS if and only if $S^K$ is a HIS.   
\end{proposition}

Given a Clifford semigroup $[Y;G_{\alpha};\psi_{\alpha,\beta}]$ then, for each $\alpha>\beta$, we let  
\begin{align*}  I_{\alpha,\beta} & :=  \text{Im }\psi_{\alpha,\beta}  =  \{a_{\beta}\in G_{\beta}:\exists a_{\alpha}\in G_{\alpha}, a_{\alpha}\psi_{\alpha,\beta}=a_{\beta}\}, \\
 K_{\alpha,\beta} & := \text{Ker }\psi_{\alpha,\beta} = \{a_{\alpha}\in G_{\alpha}:a_{\alpha}\psi_{\alpha,\beta}=e_{\beta}\},
\end{align*}  be the image and kernel of the connecting morphism $\psi_{\alpha,\beta}$, respectively. Given $\alpha>\gamma>\beta$ in $Y$ and $k_{\alpha}\in K_{\alpha,\gamma}$, then
\[ k_{\alpha}\psi_{\alpha,\beta}=(k_{\alpha}\psi_{\alpha,\gamma})\psi_{\gamma,\beta}=e_{\gamma}\psi_{\gamma,\beta}=e_{\beta},
\] 
and so $k_{\alpha}\in K_{\alpha,\beta}$. Thus $K_{\alpha,\gamma}\subseteq K_{\alpha,\beta}$, and similarly $I_{\alpha,\beta}\subseteq I_{\gamma,\beta}$.

We also define the  \textit{absolute image} $I^*_{\alpha}$ and the \textit{absolute kernel} $K_{\alpha}^*$ of $\alpha\in Y$ as the subsets of $G_{\alpha}$ given by 
\begin{align*}  I_{\alpha}^* & := \{g_{\alpha}\in I_{\alpha}: o(g_{\alpha})=o(g_{\alpha}\psi_{\alpha,\beta}) \text{ for all } \beta<\alpha\}, \\
 K_{\alpha}^* & := \{a_{\alpha}\in G_{\alpha}:a_{\alpha}\psi_{\alpha,\beta}=e_{\beta} \text{ for all } \beta<\alpha\} = \bigcap_{\beta<\alpha}K_{\alpha,\beta}.
\end{align*}

The set $K_{\alpha}^*$, being an intersection of subgroups of $G_{\alpha}$, forms a subgroup, while $I_{\alpha}^*$ may not. 

\begin{notation} {\em Throughout the remainder of this subsection we let $S=[Y; G_{\alpha}; \psi_{\alpha, \beta}]$ denote a Clifford HIS, so that $Y$ is a homogeneous semilattice and the $G_{\alpha}$ are pairwise isomorphic homogeneous groups.}
\end{notation} 

The following lemma will be vital in our understanding of the images and kernels of the connecting morphisms.

\begin{lemma} \label{iso bit}  Let $\alpha,\alpha',\beta\in Y$ be such that $\alpha,\alpha'>\beta$, and let $g_{\beta}, h_{\beta}\in G_{\beta}$ be of the same order. Then  the map 
\[ \phi: \langle e_{\alpha},g_{\beta} \rangle_I \rightarrow \langle e_{\alpha'},h_{\beta} \rangle_I 
\] 
given by $e_{\alpha}\phi=e_{\alpha'}$ and $g_{\beta}^z\phi=h_{\beta}^z$ for $z\in \mathbb{Z}$, is an isomorphism. 
\end{lemma} 

\begin{proof}
Note that $\langle g_{\beta} \rangle_I$ and $\langle h_{\beta} \rangle_I$ are isomorphic cyclic groups. Moreover, $e_{\alpha}$ is the identity in $\langle e_{\alpha},g_{\beta} \rangle_I $ since $e_{\alpha}g_{\beta}=(e_{\alpha}\psi_{\alpha,\beta})(g_{\beta}\psi_{\beta,\beta})=e_{\beta}g_{\beta}=g_{\beta}=g_{\beta}e_{\alpha}$ and so 
\[ \langle e_{\alpha},g_{\beta} \rangle_I = \{ e_{\alpha} \} \cup \langle g_{\beta} \rangle_I.
\] 
 Similarly for $\langle e_{\alpha'},h_{\beta} \rangle_I$, and it is routine to check that $\phi$ is indeed an isomorphism. 
\end{proof}
 
\begin{lemma} \label{im same} Let $\alpha,\alpha',\beta\in Y$ be such that $\alpha,\alpha'>\beta$. Then $I_{\alpha,\beta}=I_{\alpha',\beta}$.
\end{lemma} 

\begin{proof} Let $g_{\beta}=g_{\alpha}\psi_{\alpha,\beta} \in I_{\alpha,\beta}$. By the lemma above, there is an isomorphism $\phi:\langle e_{\alpha},g_{\beta} \rangle_I \rightarrow \langle e_{\alpha'},g_{\beta} \rangle_I$ determined by $e_{\alpha}\phi=e_{\alpha'}$ and $g_{\beta}\phi=g_{\beta}$. Extend $\phi$ to an automorphism ${\theta}=[{\theta}_{\alpha}, \pi]_{\alpha\in Y}$ of $S$, so that $\alpha \pi=\alpha'$ and $\beta\pi=\beta$. Then 
\[  g_{\alpha}\theta_{\alpha} \psi_{\alpha',\beta} = g_{\alpha} \psi_{\alpha,\beta} {\theta}_{\beta}  = g_{\beta}{\theta}_{\beta}= g_{\beta}
\] 
since the diagram $[\alpha,\beta;\alpha',\beta]$ commutes. Hence $g_{\beta}\in I_{\alpha',\beta}$ and $I_{\alpha,\beta} \subseteq I_{\alpha',\beta}$. The dual gives equality. 
\end{proof} 

For each $\alpha\in Y$, we let $I_{\alpha}$ denote the subgroup $I_{\delta,\alpha}$ for (any) $\delta>\alpha$. Since $Y$ has no maximal elements, $I_{\alpha}$ is non-empty for all $\alpha\in Y$.  

\begin{lemma}\label{I char} For each $\alpha\in Y$, the subgroups $I_{\alpha}$ and $K^*_{\alpha}$ are characteristic subgroups of $G_{\alpha}$, and are thus homogeneous. Moreover, for each $\alpha>\beta$, the subgroup $K_{\alpha,\beta}$ is homogeneous. 
\end{lemma} 

\begin{proof} Let $\varphi\in \text{Aut}(G_{\alpha})$ and $g_{\alpha}=g_{\delta}\psi_{\delta,\alpha}\in I_{\alpha}$ for some $\delta>\alpha$. Then, by Lemma \ref{iso bit}, we may take an isomorphism $\phi:\langle e_{\delta},g_{\alpha} \rangle_I \rightarrow \langle e_{\delta},g_{\alpha}\varphi \rangle_I$ fixing $e_{\delta}$ and with $g_{\alpha}\phi=g_{\alpha}\varphi$. Extending $\phi$ to  ${\theta}=[{\theta}_{\alpha}, \pi]_{\alpha\in Y} \in \text{Aut}(S)$, then as $[\delta,\alpha;\delta,\alpha]$ commutes, we have 
\[g_{\alpha}\varphi = g_{\alpha}\phi  = g_{\alpha}{\theta}_{\alpha} = g_{\delta}\psi_{\delta,\alpha}{\theta}_{\alpha} =    g_{\delta}{\theta}_{\delta}\psi_{\delta,\alpha} \in I_{\alpha}, 
\] 
and so $I_{\alpha}$ is characteristic.  
 Now let $k_{\alpha}\in K^*_{\alpha}$, and extend the isomorphism between $\langle k_{\alpha}\varphi \rangle_I$ and $\langle k_{\alpha} \rangle_I$ which sends $k_{\alpha}\varphi$ to $k_{\alpha}$, to $\bar{\theta}=[\bar{\theta}_{\alpha},\bar{\pi}]_{\alpha\in Y}\in \text{Aut}(S)$. Then as $[\alpha,\beta;\alpha,\beta\bar{\pi}]$ commutes for any $\beta<\alpha$, and as $k_{\alpha}\in {K}_{\alpha,\beta\bar{\pi}}$, we have 
 \[ (k_{\alpha}\varphi)\psi_{\alpha,\beta}\theta_{\beta}= (k_{\alpha}\varphi)\theta_{\alpha}\psi_{\alpha,\beta\bar{\pi}}=k_{\alpha}\psi_{\alpha,\beta\bar{\pi}} = e_{\beta\bar{\pi}}.
 \] 
 Hence $k_{\alpha}\varphi\in K_{\alpha,\beta}$ for all $\beta<\alpha$, that is, $k_{\alpha}\varphi\in K^*_{\alpha}$, and so $K_{\alpha}^*$ is characteristic. 
 
Finally, let $\phi$ be an isomorphism between f.g. subgroups $A_{\alpha}$ and $A_{\alpha}'$ of $K_{\alpha,\beta}$. Then the map  $\phi': A_{\alpha}\cup \{e_{\beta}\} \rightarrow A_{\alpha}'\cup \{e_{\beta}\}$ such that $A_{\alpha}\phi'=A_{\alpha}\phi$ and $e_{\beta}\phi'=e_{\beta}$ is an isomorphism between f.g. inverse subsemigroups of $S$. By extending $\phi'$ to an automorphism of $S$, the result follows from Theorem \ref{iso clifford}.  
\end{proof}

\begin{lemma}\label{inj iff inj} If $\alpha>\beta$ and $\alpha'>\beta'$ in $Y$ then there exists a pair of isomorphisms $\theta_{\alpha}:G_{\alpha}\rightarrow G_{\alpha'}$ and $\theta_{\beta}:G_{\beta}\rightarrow G_{\beta'}$ such that $\psi_{\alpha,\beta} = \theta_{\alpha}\psi_{\alpha',\beta'}\theta_{\beta}^{-1}$. In particular, if $\psi_{\alpha,\beta}$ is injective/surjective then so is $\psi_{\alpha',\beta'}$, and $I_{\beta}\cong I_{\beta'}$, $K_{\alpha,\beta}\cong K_{\alpha',\beta'}$ and $K^*_{\alpha} \cong K^*_{\alpha'}$.  
\end{lemma} 

\begin{proof}
Clearly the map $\phi:\{e_{\alpha},e_{\beta}\} \rightarrow \{ e_{\alpha'},e_{\beta'} \} $ given by $e_{\alpha}\phi=e_{\alpha'}$ and $e_{\beta}\phi=e_{\beta'}$ is an isomorphism. By extending $\phi$ to an automorphism $\theta=[\theta_{\alpha},\pi]_{\alpha\in Y}$ of $S$, the first result follows immediately from the commutativity of $[\alpha,\beta;\alpha\pi,\beta\pi]=[\alpha,\beta;\alpha',\beta']$. The injective/surjective properties of the connecting morphisms also follow. We observe that 
\[ I_{\beta}\theta_{\beta}=(G_{\alpha}\psi_{\alpha,\beta})\theta_{\beta} = G_{\alpha}\theta_{\alpha}\psi_{\alpha',\beta'} = I_{\beta'}
\] 
so that $I_{\beta}\cong I_{\beta'}$. If $k_{\alpha}\in K_{\alpha,\beta}$ then 
\[ k_{\alpha}\theta_{\alpha} \psi_{\alpha',\beta'}=k_{\alpha}\psi_{\alpha,\beta}\theta_{\beta}=e_{\beta}\theta_{\beta}=e_{\beta'},
\]
so that $K_{\alpha,\beta}\theta \subseteq K_{\alpha',\beta'}$. If $x_{\alpha'}\in K_{\alpha',\beta'}$, then there exists $y_{\alpha}\in G_{\alpha}$ with $y_{\alpha}\theta_{\alpha}=x_{\alpha'}$, so that
\[y_{\alpha} \psi_{\alpha,\beta} \theta_{\beta}= x_{\alpha'}\psi_{\alpha',\beta'}=e_{\beta'}. 
\]
Hence $y_{\alpha}\psi_{\alpha,\beta}=e_{\beta}$, and so $y_{\alpha}\in K_{\alpha,\beta}$. We have thus shown that $K_{\alpha,\beta}\theta = K_{\alpha',\beta'}$, and so $K_{\alpha,\beta} \cong K_{\alpha',\beta'}$. Finally, 
\[ K^*_{\alpha}\theta_{\alpha}=(\bigcap_{\gamma<\alpha} K_{\alpha,\gamma})\theta_{\alpha} = \bigcap_{\gamma<\alpha} (K_{\alpha,\gamma}\theta_{\alpha})=\bigcap_{\gamma\pi<\alpha} K_{\alpha',\gamma\pi} = K_{\alpha^*}
\] since $\pi$ is an automorphism of $Y$. Thus $K^*_{\alpha}\cong K^*_{\alpha'}$ as required. 

\end{proof}

We say that a subset $A$ of a group $G$ is \textit{closed under prime powers} if, whenever $p\in o(A)$ for some prime $p$, then every power of $p$ in $o(G)$ also lies $o(A)$. 

\begin{lemma}\label{triv int} The subgroups $I_{\alpha}$ and $K^*_{\alpha}$ are closed under prime powers and $I_{\alpha}\cap K^*_{\alpha}=\{e_{\alpha}\}$. Moreover, every element in $G_{\alpha}$ of prime order is contained in either $I_{\alpha}$ or $K^*_{\alpha}$. 
\end{lemma} 

\begin{proof} Let $p\in o(K^*_{\alpha})$. Proceeding by induction, suppose $p,p^2,\dots,p^{r-1}\in o(K^*_{\alpha})$, so that every element of $G_{\alpha}$ of order $p^k$ for $1\leq k \leq r-1$ is in $K^*_{\alpha}$ by Lemma \ref{order els}. Let $g_{\alpha}\in G_{\alpha}$ be of order $p^r$. Then $g_{\alpha}^p$ is of order $p^{r-1}$, so that $g_{\alpha}^p\in K^*_{\alpha}$. In particular, for any $\beta<\alpha$ we have $(g_{\alpha}\psi_{\alpha,\beta})^p=e_{\beta}$. If $o(g_{\alpha}\psi_{\alpha,\beta})=p$ then for any $\alpha'\in Y$ with $\alpha>\alpha'>\beta$ we have $o(g_{\alpha}\psi_{\alpha,\alpha'})=p$ and thus $g_{\alpha}\psi_{\alpha,\alpha'}\in K^*_{\alpha'}$ (as $K_{\alpha'}^* \cong K_{\alpha}^*$). Hence $(g_{\alpha}\psi_{\alpha,\alpha'})\psi_{\alpha',\beta} = g_{\alpha}\psi_{\alpha,\beta}=e_{\beta}$, a contradiction, and so $g_{\alpha}\in K^*_{\alpha}$. The inductive step is thus complete, and so $K^*_{\alpha}$ is closed under prime powers.  

Suppose for contradiction that $p\in o(I_{\alpha})\cap o(K^*_{\alpha})$ for some prime $p$, so that $I_{\alpha}\cup K_{\alpha}^*$ contains every element of $G_{\alpha}$ of order $p$ by Lemma \ref{order els}. Let $g_{\alpha}\in G_{\alpha}$ be of order $p$, so that if $\delta>\alpha$ then there exists $g_{\delta}\in G_{\delta}$ with $g_{\delta}\psi_{\delta,\alpha}=g_{\alpha}$. Suppose first that $o(g_{\delta})=p^nm$ is finite, where hcf$(p^n,m)=1$. Then $g_{\delta}^m\psi_{\delta,\alpha}=g_{\alpha}^m$ has order $p$, while $g_{\delta}^m$ has order $p^n$. Since $K^*_{\delta}$ is closed under prime order we have $g_{\delta}^m\in K_{\delta}^*$, a contradiction. It follows that the pre-images of elements of order $p$ under the connecting morphisms are of infinite order. Let $\delta>\tau>\alpha$ and let $h_{\tau}\in G_{\tau}$ be of order $p$.
 The map from $\langle e_{\delta},g_{\alpha} \rangle_I$ to $\langle e_{\delta},h_{\tau} \rangle_I$ which fixes $e_{\delta}$ and sends $g_{\alpha}$ to $h_{\tau}$ is an isomorphism by Lemma \ref{iso bit}. By extending it to an automorphism $\theta$ of $S$, we have that $g_{\delta}\theta=h_{\delta}\in G_{\delta}$ is such that $h_{\delta}\psi_{\delta,\tau}=h_{\tau}$, so that $o(h_{\delta})=\aleph_0$ and $h_{\delta}\in K_{\delta,\alpha}$. Since $g_{\delta}^p \in K_{\delta,\alpha}$, it follows that the map between $\langle h_{\delta},e_{\alpha} \rangle_I$ and $\langle g_{\delta}^p,e_{\alpha}\rangle_I$ which fixes $e_{\alpha}$ and sends $h_{\delta}$ to $g_{\delta}^p$ is an isomorphism, and we may thus extend it to an automorphism $[\theta_{\alpha},\pi]_{\alpha\in Y}$  of $S$. Then $\delta=\delta\pi>\tau\pi>\alpha\pi=\alpha$ and, by the commutativity of $[\delta,\tau;\delta,\tau\pi]$, 
\[ h_{\tau}\theta_{\tau} = h_{\delta}\psi_{\delta,\tau}\theta_{\tau} = h_{\delta}\theta_{\delta}\psi_{\delta,\tau\pi} = g_{\delta}^p\psi_{\delta,\tau\pi}.
\] 
 Hence $g_{\delta}^p\psi_{\delta,\tau\pi}$ is of order $p$, however $(g_{\delta}\psi_{\delta,\tau\pi})\psi_{\tau\pi,\alpha} = g_{\alpha}$, so that $g_{\delta}\psi_{\delta,\tau\pi}$ is of infinite order, and a contradiction is achieved.  
 
 Now suppose $x_{\alpha}\in I_{\alpha}\cap K^*_{\alpha}$ has order $n\in\mathbb{N}^*$. If $n$ is finite, then there exists a prime $p$ with $p|n$, and so $x_{\alpha}^{n/p}\in I_{\alpha}\cap K^*_{\alpha}$ has order $p$, a contradiction. If $n$ is infinite then there exists $\delta>\alpha$ and $x_{\delta}\in G_{\delta}$ with $x_{\delta}\psi_{\delta,\alpha}=x_{\alpha}$, so that $x_{\delta}$ is of infinite order. Since the absolute kernels are pairwise isomorphic we have $\aleph_0\in o(K_{\sigma}^*)$ for each $\sigma\in Y$. Hence $K^*_{\delta}$ contains every element of infinite order in $G_{\delta}$ by Lemma \ref{order els}, and so $x_{\delta}\psi_{\delta,\alpha}=e_{\alpha}$, a contradiction. Hence $I_{\alpha}$ and $K^*_{\alpha}$ have trivial intersection.
 
 We may now prove that $I_{\alpha}$ is closed under prime powers. Let $p\in o(I_{\alpha})$ for some prime $p$, and let $z_{\alpha}\in G_{\alpha}$ be of order $p^r$. If $o(z_{\alpha}\psi_{\alpha,\beta})<p^r$ for all $\beta<\alpha$ then $z_{\alpha}^{p^{r-1}}\in I_{\alpha} \cap K^*_{\alpha}$, a contradiction. Hence there exists $\beta$ with $z_{\alpha}\psi_{\alpha,\beta}$ of order $p^r$, so that $p^r\in o(I_{\beta})$. By Lemma \ref{inj iff inj} $o(I_{\beta})=o(I_{\alpha})$ and so $I_{\alpha}$ is closed under prime powers. 
 
 Finally, let $a_{\alpha}\in G_{\alpha}$ be of prime order $p$. If $a_{\alpha}\not\in K^*_{\alpha}$ then $a_{\alpha}\psi_{\alpha,\beta}$ has order $p$ for some $\beta<\alpha$, and so by the usual argument $p\in o(I_{\alpha})$, thus completing the final result.  
\end{proof} 

Consequently, by Corollary \ref{coprime char subgroups}, the subgroups $I_{\alpha}$ and $K^*_{\alpha}$ are of coprime order. Hence, as $I_{\alpha}$ and $K^*_{\alpha}$  are characteristic subgroups of $G_{\alpha}$, they are invarient under inner automorphisms of $G_{\alpha}$ and thus are normal, we have $\langle I_{\alpha},K^*_{\alpha}\rangle_I=I_{\alpha}\otimes K^*_{\alpha}$. 

\begin{lemma}\label{coprime} If $G_{\alpha}$ is periodic then $G_{\alpha}=I_{\alpha}\otimes K^*_{\alpha}$. If $G_{\alpha}$ is non-periodic then either $G_{\alpha}=I_{\alpha}$ or $G_{\alpha}=K^*_{\alpha}$ or $I_{\alpha} \otimes K^*_{\alpha}$ is the precisely set of elements of finite order in $G_{\alpha}$. 
\end{lemma} 

\begin{proof} 
If $g_{\alpha}\in G_{\alpha}$ has finite order $n=p_1^{n_1}\dots p_r^{n_r}$ for some primes $p_i$ then, by the Fundamental Theorem of Finite Abelian Groups, $g_{\alpha}=g_{\alpha,1}g_{\alpha,2}\dots g_{\alpha,r}$ for some $g_{\alpha,i}\in G_{\alpha}$ of order $p_i^{n_i}$. By the previous corollary we have $g_{\alpha,i}\in I_{\alpha}\cup K^*_{\alpha}$, and so $g_{\alpha}\in I_{\alpha}\otimes K^*_{\alpha}$. Consequently, the subgroup $I_{\alpha}\otimes K_{\alpha}^*$ contains every element of $G_{\alpha}$ of finite order, and so if $G_{\alpha}$ is periodic then $G_{\alpha}=I_{\alpha}\otimes K_{\alpha}^*$. 

Now suppose that $G_{\alpha}$ contains an element $x_{\alpha}$ of infinite order.  Suppose first that  $x_{\alpha}\in I_{\alpha}\otimes K^*_{\alpha}$, say $x_{\alpha}=g_{\alpha}k_{\alpha}$. Then either $g_{\alpha}$ or $k_{\alpha}$ has infinite order, as $I_{\alpha}$ and $K^*_{\alpha}$ are of coprime order. Hence, by Lemma \ref{order els}, either $I_{\alpha}$ or $K^*_{\alpha}$ contains all elements of infinite order, and so $G_{\alpha}=I_{\alpha}\otimes K^*_{\alpha}$. If $g_{\alpha}$ is of infinite order, then for any $m_{\alpha}\in K^*_{\alpha}$ we have that $g_{\alpha}m_{\alpha}$ has infinite order. Hence $g_{\alpha}^{-1}$ and $g_{\alpha}m_{\alpha}$, being of infinite order, are in $I_{\alpha}$, and thus so is $g^{-1}_{\alpha}(g_{\alpha}m_{\alpha})=m_{\alpha}$. Thus $m_{\alpha}=e_{\alpha}$, and so $G_{\alpha}=I_{\alpha}$. The case where $k_{\alpha}$ is of infinite order is proven similarly. 

If no such element $x_{\alpha}$ exists, then $I_{\alpha}\otimes K_{\alpha}$ is precisely the elements of finite order as required. 

\end{proof} 

We may say more about the final case in Lemma \ref{coprime}, in particular that each maximal subgroup is the union of its kernel subgroups: 

\begin{lemma} \label{absolute image trivial} If $G_{\alpha}$ is non-periodic and $I_{\alpha}\otimes K^*_{\alpha}$ periodic, then the inverse subsemigroup $[Y;I_{\alpha}\otimes K^*_{\alpha},\psi_{\alpha,\beta}|_{I_{\alpha}\otimes K^*_{\alpha}}]$ of $S$ is a HIS. Moreover, the absolute image $I_{\alpha}^*$ of $G_{\alpha}$ is trivial and  $G_{\alpha}=\bigcup_{\beta<\alpha} K_{\alpha,\beta}$. 
\end{lemma}   

\begin{proof} The first result is immediate from Lemma \ref{char sub his} since $I_{\alpha}\otimes K^*_{\alpha}$, being the subgroup containing all periodic elements of $G_{\alpha}$, is order-characteristic.  

We claim that each element $x_{\alpha}$ of infinite order in $G_{\alpha}$ is contained in the kernel of some connecting morphism. For any $\beta<\alpha$ we have that $x_{\alpha}\psi_{\alpha,\beta}$ has finite order, say $n$, since $\aleph_0 \not\in o(I_{\alpha})$. Hence $x_{\alpha}^n$ is an element of infinite order with $x_{\alpha}^n\in K_{\alpha,\beta}$. The claim easily follows by homogeneity. 

Now suppose that $g_{\alpha}\in I_{\alpha}^*$. Then $x_{\alpha}g_{\alpha}$ has infinite order, since otherwise $x_{\alpha}g_{\alpha}$ is an element of $I_{\alpha}\otimes K^*_{\alpha}$, and thus so is $x_{\alpha}=(x_{\alpha}g_{\alpha})g_{\alpha}^{-1}$. By the previous claim, $x_{\alpha}\in K_{\alpha,\beta}$ and $x_{\alpha}g_{\alpha}\in K_{\alpha,\gamma}$ for some $\beta,\gamma<\alpha$. Hence  
\[ (x_{\alpha}g_{\alpha})\psi_{\alpha,\beta\gamma}= (x_{\alpha}\psi_{\alpha,\beta}\psi_{\beta,\beta\gamma})(g_{\alpha}\psi_{\alpha,\beta\gamma}) = g_{\alpha}\psi_{\alpha,\beta\gamma} = e_{\beta\gamma} 
\] 
and so $g_{\alpha}=e_{\alpha}$. Hence $I_{\alpha}^*$ is trivial as required. 

Finally, if there exists $z_{\alpha}\notin \bigcup_{\beta<\alpha} K_{\alpha,\beta}$ then $z_{\alpha}\psi_{\alpha,\beta}\in I_{\alpha}^*$  for some $\beta<\alpha$, a contradiction, and the final result is obtained. 
\end{proof}

We call a Clifford semigroup in which each connecting morphism is surjective a \textit{surjective Clifford semigroup}.

\begin{corollary} \label{surj trivial ak} Let $T=[Y;H_{\alpha};\phi_{\alpha,\beta}]$ be a surjective Clifford HIS. Then the absolute kernels of $T$ are trivial.
\end{corollary}

\begin{proof}
Immediate from Lemma \ref{triv int}. 
\end{proof}

\begin{lemma} The inverse subsemigroup $[Y;I_{\alpha};\psi_{\alpha,\beta}^I]$ of $S$ is a surjective Clifford semigroup, where $\psi_{\alpha,\beta}^I=\psi_{\alpha,\beta}|_{I_{\alpha}}$.
\end{lemma}

\begin{proof} By definition $I_{\alpha}\psi_{\alpha,\beta}\subseteq I_{\beta}$. Let $x_{\beta}\in I_{\beta}=$. If $G_{\alpha}$ is periodic then by Lemma \ref{coprime} (and as $I_{\beta}=$ Im $\psi_{\alpha,\beta}$) there exists $g_{\alpha}\in I_{\alpha}$ and $k_{\alpha}\in K_{\alpha}^*$ such that $(g_{\alpha}k_{\alpha})\psi_{\alpha,\beta}=x_{\beta}$. Hence $g_{\alpha}\psi_{\alpha,\beta}=x_{\beta}$ and  $\psi^I_{\alpha,\beta}$ is surjective. 
If $G_{\alpha}$ is non-periodic, then the result is trivially true when $G_{\alpha}=I_{\alpha}$, or when $G_{\alpha}=K^*_{\alpha}$, since in this case $I_{\alpha}=\{e_{\alpha}\}$, so we suppose $\aleph_0 \not\in o(I_{\alpha})\cup o(K^*_{\alpha})$. Then as $[Y;I_{\alpha}\otimes K^*_{\alpha};\psi_{\alpha,\beta}|_{I_{\alpha}\otimes K^*_{\alpha}}]$ forms a periodic HIS by the previous lemma, the result follows by the periodic case. 
\end{proof}

For each $\alpha>\beta$ in $Y$, we let $\tau_{\alpha,\beta}$ denote the trivial morphism from $K^*_{\alpha}$ to $K^*_{\beta}$, and let $\tau_{\alpha,\alpha}=1_{K^*_{\alpha}}$. We shall call a Clifford semigroup in which each connecting morphism is trivial an \textit{image-trivial Clifford semigroup}. 

We observe that, for each $\alpha\geq \beta$,  $g_{\alpha}\in I_{\alpha}$ and $k_{\alpha}\in K^*_{\alpha}$ then $(g_{\alpha}k_{\alpha})\psi_{\alpha,\beta}= (g_{\alpha}\psi_{\alpha,\beta}^I)(k_{\alpha}\tau_{\alpha,\beta})$. 
Hence  $S$ has two important inverse subsemigroups, 
\[ I(S):=[Y;I_{\alpha};\psi_{\alpha,\beta}^I] \text{  and  } K(S):=[Y;K^*_{\alpha};\tau_{\alpha,\beta}], 
\] 
which are HIS by Lemma \ref{char sub his}. Moreover, if $S$ is periodic then $S=[Y;I_{\alpha}\otimes K_{\alpha};\psi_{\alpha,\beta}^I \otimes \tau_{\alpha,\beta}]$.

We summarise the findings thus far in this section as follows:

\begin{theorem}\label{first thm} Let $S$ be a periodic Clifford HIS. Then $S=[Y;I_{\alpha}\otimes K_{\alpha}^*;\psi_{\alpha,\beta}^I \otimes \tau_{\alpha,\beta}]$, where: 
\begin{enumerate}[label=(\roman*), font=\normalfont]
\item $Y$ is a homogeneous semilattice; 
\item $I(S)=[Y;I_{\alpha};\psi^I_{\alpha,\beta}]$ is a surjective Clifford HIS; 
\item $K(S)=[Y;K_{\alpha}^*;\tau_{\alpha,\beta}]$ is an image-trivial Clifford HIS; 
\item there exists a homogeneous group $G=I\otimes K^*$ where $I$ and $K^*$ are of coprime order,  such that $G\cong G_{\alpha}$, $I_{\alpha}\cong I$ and $K_{\alpha}^*\cong K^*$ for all $\alpha\in Y$.
  \end{enumerate}  
A non-periodic Clifford semigroup is a HIS if and only if it is isomorphic to either a surjective Clifford HIS, or an image-trivial Clifford HIS, or a Clifford HIS with no elements of infinite order lying in the images or absolute kernels. 
\end{theorem} 

In the next subsection we shall prove a converse to the first result of Theorem \ref{first thm}. This relies on proving the stronger property of homogeneity for image-trivial Clifford semigroups. 

\begin{lemma}\label{homo iso} Let $G$ and $H$ be isomorphic homogeneous groups. Then any isomorphism between f.g. subgroups $A$ and $B$ of $G$ and $H$, respectively, can be extended to an isomorphism between $G$ and $H$. 
\end{lemma} 

\begin{proof} Let $\theta:A\rightarrow B$ be an isomorphism and consider any isomorphism $\phi:G\rightarrow H$. Then $\theta \, \phi^{-1}|_{B}:A\rightarrow B \phi^{-1}$ is an isomorphism between f.g. subgroups of $G$, which can thus be extended to an automorphism $\chi$ of $G$. The isomorphism $\chi \phi:G\rightarrow H$ extends $\theta$, since if $g\in A$ then
\[ g \chi \phi = g (\theta \, \phi^{-1}|_{B}) \, \phi = g\theta.
\] 
\end{proof}

\begin{lemma}\label{trivial mor} Let $T=[Y;G_{\alpha};\tau_{\alpha,\beta}]$ be an image-trivial Clifford semigroup. Then the following are equivalent: 
\begin{enumerate}[label=(\roman*), font=\normalfont]
\item  $T$ \text{is a structure-HIS}; 
\item  $T$ \text{is a HIS}; 
\item  $Y$ is homogeneous and there exists a homogeneous group $G$  such that $G\cong G_{\alpha}$ for all $\alpha\in Y$.
  \end{enumerate} 

\end{lemma} 

\begin{proof} $\mathrm{i}) \Rightarrow \mathrm{ii})$ Immediate, as every structure-HIS is a HIS. 

$\mathrm{ii}) \Rightarrow \mathrm{iii})$ Immediate from Corollary \ref{structure clifford}.

$\mathrm{iii}) \Rightarrow \mathrm{i})$ Let $A_1=[Z';A_{\gamma}';\tau'_{\gamma,\delta}]$ and $A_2=[Z'';A_{\eta}'';\tau''_{\eta,\sigma}]$ be a pair of f.g. inverse subsemigroups of $T$, where the maps $\tau'_{\gamma,\delta}$ and $\tau''_{\eta,\sigma}$, being restrictions of  trivial morphisms, are trivial.
 Let $\theta=[\theta_{\gamma},\pi]_{\gamma\in Z_1}$ be an isomorphism from $A_1$ to $A_2$ and let $\bar{\pi}$ be an automorphism of $Y$ which extends $\pi$. By Lemma \ref{homo iso} we may extend each $\theta_{\gamma}:A_{\gamma}' \rightarrow A_{\gamma\pi}''$ to an isomorphism $\bar{\theta}_{\gamma}:G_{\gamma}\rightarrow G_{\gamma\pi}$. For each $\alpha\not\in Z_1$, let $\bar{\theta}_{\alpha}$ be any isomorphism from $G_{\alpha}$ to $G_{\alpha\bar{\pi}}$. We claim that $\bar{\theta}=[\bar{\theta}_{\alpha},\bar{\pi}]_{\alpha\in Y}$ is an automorphism of $S$. Indeed, for any $g_{\alpha}\in G_{\alpha}$ and $\alpha> \beta$ we have 
\[ g_{\alpha} \bar{\theta}_{\alpha}\tau_{\alpha\bar{\pi},\beta\bar{\pi}} = e_{\beta \bar{\pi}}= e_{\beta}\bar{\theta}_{\beta} = g_{\alpha}\tau_{\alpha,\beta} \bar{\theta}_{\beta}
\] 
and so the diagram $[\alpha,\beta;\alpha\bar{\pi},\beta\bar{\pi}]$ commutes and thus $\bar{\theta}$ is indeed an automorphism. Moreover, by construction $\bar{\pi}$ extends $\pi$, and so $\bar{\theta}$ extends $\theta$. Hence $T$ is a structure-HIS. 
\end{proof} 

 Let $T=[Y;G_{\alpha};\tau_{\alpha,\beta}]$ be an image-trivial Clifford semigroup such that $G_{\alpha}\cong G$ for each $\alpha\in Y$.  Let $\theta_{\alpha}:G_{\alpha}\rightarrow G$ be an isomorphism for each  $\alpha\in Y$, and define a bijection $\theta:T \rightarrow Y\times G$ by 
\[ g_{\alpha}\theta= (\alpha,g_{\alpha}\theta_{\alpha}), 
\] 
for each $g_{\alpha}\in G_{\alpha}$, $\alpha\in Y$. Then we may use $\theta$ to endow the set $Y\times G$ with a multiplication
\begin{align*}
 (\alpha,g)*(\beta,h) =
  \begin{cases}
    (\alpha, gh) & \text{if } \alpha=\beta,\\ 
   (\beta,h) & \text{if } \alpha>\beta, \\ 
   (\alpha,g) & \text{if } \alpha<\beta, \\
    (\alpha\beta,1) & \text{if } \alpha \bot \beta.   \end{cases} 
\end{align*}
We denote the resulting semigroup $(Y\times G,*)$ as $[Y;G]$. We have thus shown that: 

\begin{lemma}\label{trivial iso} Let $T=[Y;G_{\alpha};\tau_{\alpha,\beta}]$ be an image-trivial Clifford semigroup such that $G_{\alpha}\cong G$ for each $\alpha\in Y$. Then $T\cong [Y;G]$.  
\end{lemma} 

\subsection{Spined product}

We now consider a more succinct form of a periodic Clifford HIS semigroup $[Y;I_{\alpha}\otimes K^*_{\alpha};\psi_{\alpha,\beta}^I \otimes \tau_{\alpha,\beta}]$ using \textit{spined products}.  

\begin{definition}\cite{Ciric} {\em Given a pair of semigroups $P$ and $Q$, with common morphic image $M$, then the \textit{spined product of $P$ and $Q$ w.r.t. $M$} is the subsemigroup 
\[ \{ (a,b): a\in P, b\in Q, a\sigma=b\sigma'\} 
\] 
of $P\times Q$, where $\sigma:P\rightarrow M$ and $\sigma':Q\rightarrow M$ are morphisms onto $M$.}
\end{definition} 

Let $S_i=[Y;G_{\alpha}^{(i)};\psi_{\alpha,\beta}^{(i)}]$ ($i=1,2$) be a pair of Clifford semigroups. Then the map $\sigma_i:S_i\rightarrow Y$ such that $G_\alpha^{(i)}\sigma_i=\alpha$ for each $\alpha\in Y$ is a morphism (and is the natural map associated with the congruence $\mathcal{H}$), and the spined product of $S_1$ and $S_2$ with respect to $Y$ is then given by 
\[   \{(a_{\alpha},b_{\alpha}): a_{\alpha} \in S_1, b_{\alpha} \in S_2, \alpha\in Y\},
\] 
which we denote as $S_1 \bowtie S_2$.

\begin{lemma}\label{spined to cliff} Let $S_1$ and $S_2$ be defined as above. 
Then $ S_1 \bowtie S_2$ is isomorphic to the Clifford semigroup  $S=[Y;G_{\alpha}^{(1)} \otimes G_{\alpha}^{(2)};\psi_{\alpha,\beta}^{(1)}\otimes \psi_{\alpha,\beta}^{(2)}]$. 
\end{lemma} 
\begin{proof}  The map $\phi: S_1 \bowtie S_2  \rightarrow  S$ given by $(g_{\alpha},h_{\alpha})\phi=g_{\alpha}h_{\alpha}$ for each $(g_{\alpha},h_{\alpha})\in G^{(1)}_{\alpha}\otimes G^{(2)}_{\alpha}$ is easily shown to be an isomorphism.
\end{proof} 

\begin{note} {\em Let $S_i=[Y;G^{(i)}_{\alpha};\psi^{(i)}_{\alpha,\beta}]$ and $S_i'=[Y';H^{(i)}_{\alpha'};\phi^{(i)}_{\alpha',\beta'}]$ be Clifford semigroups ($i=1,2$) 
 and consider a pair of isomorphisms 
\[ \theta^{(i)}=[\theta^{(i)}_{\alpha},\pi]_{\alpha\in Y}:S_i \rightarrow S_i'  \quad (i=1,2).
\] 
 Then the map $\theta:S_1\bowtie S_2 \rightarrow S_1'\bowtie S_2'$ given by $(g_{\alpha},h_{\alpha})\theta=(g_{\alpha}\theta_{\alpha}^{(1)},h_{\alpha}\theta_{\alpha}^{(2)})$ is clearly an isomorphism by Theorem \ref{iso clifford}, which we denote as $\theta^{(1)}\bowtie \theta^{(2)}$. Note that the induced semilattice isomorphisms of $\theta^{(1)}$ and $\theta^{(2)}$ are required to be equal in this construction. }
\end{note} 

\begin{corollary}\label{iso spined} Let $S=S_1\bowtie S_2$ and $S'=S_1'\bowtie S_2'$ be a pair of spined products of Clifford semigroups such that $S_2$ and $S_2'$ are structure-HIS's. Then $S\cong S'$ if $S_1\cong S_1'$ and $S_2\cong S_2'$. 
\end{corollary}

\begin{proof} Let $S$ have structure semilattice $Y$, and $\theta^{(1)}=[\theta_{\alpha}^{(1)},\pi]_{\alpha\in Y}$ be an isomorphism from $S_1$ to $S_1'$ and $\theta^{(2)}=[\theta_{\alpha}^{(2)},\hat{\pi}]_{\alpha\in Y}$ an isomorphism from $S_2$ to $S_2'$. Then $\pi\hat{\pi}^{-1}$ is an automorphism of $Y$, and so as $S_2$ is a structure-HIS there exists an automorphism $\phi$ of $S_2$ with induced automorphism $\pi\hat{\pi}^{-1}$. Hence $\phi\theta^{(2)}:S_2\rightarrow S_2'$ is an isomorphism, with induced isomorphism $\pi\hat{\pi}^{-1}\hat{\pi}=\pi$, and so from the note above $\theta^{(1)}\bowtie (\phi\theta^{(2)})$ is an isomorphism from $S$ to $S'$ as required. 
\end{proof}

If $S=[Y;I_{\alpha}\otimes K^*_{\alpha};\psi_{\alpha,\beta}^I \otimes \tau_{\alpha,\beta}]$ is a periodic HIS then by Lemma \ref{spined to cliff} $S$ is isomorphic to $[Y;I_{\alpha};\psi_{\alpha,\beta}^I] \bowtie [Y; K^*_{\alpha}; \tau_{\alpha,\beta}]$ and thus, by Corollary \ref{iso spined} and the structure-homogeneity of $[Y;K^*]$, to $[Y;I_{\alpha};\psi_{\alpha,\beta}^I] \bowtie [Y;K^*]$,
where $K^*\cong K^*_{\alpha}$.

We have thus proven the forward half of the periodic case of the following: 

\begin{theorem}\label{homog Thm} A periodic Clifford semigroup $S$ is a HIS if and only if there exists a homogeneous group $G=I\otimes K^*$, where $I$ and $K^*$ are of coprime order, and a surjective Clifford HIS $[Y;I_{\alpha};\psi_{\alpha,\beta}^I]$ with $I_{\alpha}\cong I$ such that 
\[ S\cong [Y;I_{\alpha};\psi_{\alpha,\beta}^I]\bowtie [Y;K^*]. 
\] 
 A non-periodic Clifford semigroup $S$ is a HIS if and only if $S$ is isomorphic to either a  surjective Clifford HIS, or $[Y;G]$ for some homogeneous semilattice $Y$ and group $G$, or a Clifford HIS with no elements of infinite order lying in the images or absolute kernels. 
\end{theorem} 

\begin{proof} We first prove the backwards half of the periodic case. Let $G$ and $Y$ be as in the hypothesis of the theorem. 
Then as $[Y;K^*]$ is structure-HIS semigroup by Lemma \ref{trivial mor}, the semigroup $[Y;I_{\alpha};\psi_{\alpha,\beta}^I]\bowtie [Y;K^*]$ is a HIS if and only if $[Y;I_{\alpha};\psi_{\alpha,\beta}^I]$ is a HIS by Proposition \ref{structure sums 2}. The non-periodic case follows immediately from Theorem \ref{first thm} and Lemma \ref{trivial mor}. 

\end{proof}

It thus suffices to consider the homogeneity of both surjective Clifford semigroups with trivial absolute kernels, and the case where there exists elements of infinite order lying outside the images and absolute kernels.

%

\subsection{Connecting isomorphisms} 
 
In this subsection we consider the homogeneity of a surjective Clifford semigroup $S$ such that a non-trivial connecting morphism is an isomorphism. By Lemma \ref{inj iff inj} this focces all connecting morphisms to be isomorphisms. In \cite{Quinn}, strong semilattices of rectangular bands with connecting morphisms being isomorphisms were considered. This work relied only on the fact that, for strong semilattices of rectangular bands, an isomorphism theorem exists of the form in Theorem \ref{iso clifford}, and not on the behaviour of the rectangular bands. As such, the proofs of Lemmas \ref{iso surj} and \ref{iso state} and Proposition \ref{isos structure hom}  are mirrored in Lemmas 4.16 and 4.18 and Proposition 4.17 of \cite{Quinn}, and as such are omitted.

\begin{lemma} \label{iso surj} Let $S=[Y;G_{\alpha};\psi_{\alpha,\beta}]$ be a Clifford semigroup such that each $\psi_{\alpha,\beta}$ is an isomorphism. Let $\pi \in \text{Aut }(Y)$ and, for a fixed $\alpha^*\in Y$, let $\theta_{\alpha^*}:G_{\alpha^*}\rightarrow G_{\alpha^*\pi}$ be an isomorphism. For each $\delta \in Y$, let $\theta_{\delta}:G_{\delta}\rightarrow G_{\delta\pi}$ be given by 
\begin{equation*} \label{iso iso} {\theta}_{\delta}= {\psi}_{\delta,\alpha^*} \, {\theta}_{\alpha^*} \, {\psi}_{\alpha^*\pi,\delta \pi}. 
\end{equation*}  
Then $\theta=[\theta_{\alpha},\pi]_{\alpha\in Y}$ is an automorphism of $S$. Conversely, every automorphism of $S$ can be so constructed. 
\end{lemma} 

%

\begin{proposition}\label{isos structure hom} Let $S=[Y;G_{\alpha};\psi_{\alpha,\beta}]$ be a Clifford semigroup such that each connecting morphism is an isomorphism. Then $S$ is a structure-HIS if and only if $Y$ and $G_{\alpha}$ are homogeneous for (any) $\alpha\in Y$. 
\end{proposition}

\begin{lemma}\label{iso state} Let $S=[Y;G_{\alpha};\psi_{\alpha,\beta}]$ be a Clifford semigroup such that each connecting morphism is an isomorphism. Then $S\cong Y \times G$ for any group $G$ isomorphic to $G_{\alpha}$. 
\end{lemma}


Hence by Proposition \ref{isos structure hom} and Corollary \ref{structure clifford}, we have that $S=Y\times G$ is a structure-HIS if and only if $S$ is a HIS if and only if $G$ and $Y$ are homogeneous.

Since all surjective morphisms between finite groups are isomorphisms, we thus have by Proposition \ref{bisimple inf} and Theorem \ref{homog Thm}: 

\begin{theorem} Given a homogeneous semilattice $Y$ and a pair of finite homogeneous groups $I$ and $K^*$ of coprime orders, the Clifford semigroup $(Y\times I)\bowtie [Y;K^*]$ is a HIS.
 Conversely, every HIS with finite maximal subgroups is isomorphic to an inverse semigroup constructed in this way. 
\end{theorem}

\subsection{Non-injective surjective Clifford semigroups} 

Throughout this subsection we let $S=[Y;G_{\alpha};\psi_{\alpha,\beta}]$ be a surjective Clifford HIS such that each $\psi_{\alpha,\beta}$ is non-injective. Recall that  the absolute kernels of $S$ are trivial by Corollary \ref{surj trivial ak}. Following in line with the general case, we shall attempt to decompose the maximal subgroups into direct products of characteristic subgroups. 

The group $G_{\alpha}$ contains two key subsets: the absolute image $I_{\alpha}^*$ and 
\begin{align*}
& T_{\alpha} = \{g_{\alpha}\in G_{\alpha}: \exists \beta< \alpha \text{ such that } g_{\alpha}\in K_{\alpha,\beta}\}= \bigcup_{\beta<\alpha}K_{\alpha,\beta}. 
\end{align*}
The set $T_{\alpha}$ forms a subgroup of $G_{\alpha}$, since if $k_{\alpha}\in K_{\alpha,\beta}$ and $m_{\alpha}\in K_{\alpha,\beta'}$ then $k_{\alpha}m_{\alpha}\in K_{\alpha,\beta\beta'}$. While $I^*_{\alpha}$ may not form a subgroup, it is closed under powers, since if $g_{\alpha}\in I^*_{\alpha}$ then for each $r\in \mathbb{N}$ and $\beta\leq \alpha$, 
\[ o(g_{\alpha})=o(g_{\alpha}\psi_{\alpha,\beta}) \Rightarrow o(g_{\alpha}^r)=o((g_{\alpha}\psi_{\alpha,\beta})^r)=o(g_{\alpha}^r\psi_{\alpha,\beta}) 
\] 
so that $g_{\alpha}^r\in I^*_{\alpha}$. 

By the usual arguments we have:  

\begin{lemma} \label{small stuff} For each $\alpha\in Y$,  $T_{\alpha}$ is a characteristic subgroup of $G_{\alpha}$ and $I^*_{\alpha}$ is a characteristic subset of $G_{\alpha}$ with $I^*_{\alpha}\cap T_{\alpha}=\{e_{\alpha}\}$.  Moreover, $T_{\alpha}\cong T_{\beta}$ and $\langle I^*_{\alpha} \rangle_I \cong \langle I^*_{\beta} \rangle_I $  for each $\alpha,\beta\in Y$.  
\end{lemma} 

Consequently, $o(I_{\alpha}^*)=o(I_{\beta}^*)$ for each $\alpha,\beta\in Y$ and $\langle I_{\alpha}^* \rangle_I$ is a characteristic subgroup of $G_{\alpha}$. Moreover, $I^*_{\alpha}$ and $T_{\alpha}$ are coprime by Corollary \ref{coprime char subgroups}. We fix the following subsets of $S$. 
\[ A(S):=[Y;\langle I^*_{\alpha} \rangle_I;\psi_{\alpha,\beta}^I] \text{ and } T(S):=[Y;T_{\alpha};\psi_{\alpha,\beta}^T], 
\] 
where $\psi_{\alpha,\beta}^I=\psi_{\alpha,\beta}|_{\langle I^*_{\alpha} \rangle_I}$ and $\psi_{\alpha,\beta}^T=\psi_{\alpha,\beta}|_{T_{\alpha}}$. 

\begin{lemma} \label{subsemigroups} For each $\alpha\in Y$, the subsets $A(S)$ and  $T(S)$ of $S$ are Clifford HIS's. Moreover, if $S$ is periodic then $A(S)$ and $T(S)$ are surjective. 
\end{lemma}   

\begin{proof} To prove that $A(S)$ and $T(S)$ are inverse subsemigroups, it suffices to show that $\psi^T_{\alpha,\beta}$ and $\psi^I_{\alpha,\beta}$ are maps to $T_{\beta}$ and $\langle I_{\beta}^* \rangle_I$, respectively. If $k_{\alpha}\in T_{\alpha}$, say, $k_{\alpha}\in K_{\alpha,\gamma}$, then $k_{\alpha}\psi_{\alpha,\beta}\in K_{\beta,\beta\gamma}\subseteq T_{\beta}$. If $g_{\alpha}\in I^*_{\alpha}$ then as $o(g_{\alpha})=o(g_{\alpha}\psi_{\alpha,\beta})$ we have $g_{\alpha}\psi_{\alpha,\beta}\in I^*_{\beta}$ by Lemma \ref{order els}, and so $\langle I_{\alpha}^* \rangle_I \psi_{\alpha,\beta}^I\subseteq \langle I_{\beta}^* \rangle_I$ as required.
 Hence $A(S)$ and $T(S)$ are inverse subsemigroups and, by Lemma \ref{char sub his}, are HISs. 

Finally, as $G_{\alpha}$ has trivial absolute kernel, so do $\langle I_{\alpha}^* \rangle_I$ and $T_{\alpha}$. Hence as $S$ is periodic then it follows from Theorem \ref{homog Thm} that $A(S)$ and $T(S)$ are surjective.
\end{proof} 

\begin{lemma}\label{kernel neq} If $\alpha>\beta>\gamma$ in $Y$ then $K_{\alpha,\beta}\subsetneq K_{\alpha,\gamma}$.
\end{lemma} 

\begin{proof} If $K_{\alpha,\beta} =  K_{\alpha,\gamma}$, then it follows by a simple application of homogeneity that $K_{\alpha,\beta}=K_{\alpha,\beta'}$ for all $\beta,\beta'<\alpha$. Hence $T_{\alpha}=K_{\alpha,\beta}$ is the absolute kernel of $G_{\alpha}$ which, being trivial, implies that each connecting morphism is injective, a contradiction. 
\end{proof} 


\begin{lemma}\label{surj splits} For each $\alpha\in Y$ we have $G_{\alpha}=T_{\alpha}I^*_{\alpha}$. Consequently, if $I^*_{\alpha}$ forms a subgroup then $G_{\alpha}=T_{\alpha}\otimes I_{\alpha}^*$ and, if in addition $G_{\alpha}$ is non-periodic, then $I_{\alpha}^*$ is trivial, so that $G_{\alpha}=T_{\alpha}$.
\end{lemma} 

\begin{proof} Let $a_{\alpha}\in G_{\alpha} \setminus (T_{\alpha}\cup I^*_{\alpha})$ have finite order $n$. Then there exists $\beta<\alpha$ such that $a_{\alpha}\psi_{\alpha,\beta}=a_{\beta}$ has order $m<n$, say, $n=mk$. We choose $\beta$ so that  $a_{\beta}$ is of minimal order, noting that $a_{\beta}\neq e_{\beta}$ as $a_{\alpha}\not\in T_{\alpha}$. Then $a_{\beta}\in I^*_{\beta}$, since if $o(a_{\beta}\psi_{\beta,\gamma})<m$ for some $\gamma<\beta$ then $o(a_{\alpha}\psi_{\alpha,\gamma})<m$, contradicting the minimality of $m$. Since $o(I_{\alpha}^*)=o(I_{\beta}^*)$, it follows from Lemma \ref{order els} that $a_{\alpha}^k\in I^*_{\alpha}$.
 Moreover $a_{\alpha}^m\psi_{\alpha,\beta}=a_{\beta}^m=e_{\beta}$, so $a_{\alpha}^m\in T_{\alpha}$. Hence as $T_{\alpha}$ is characteristic , contains all elements of order $k$.  Hence as $T_{\alpha}$ is characteristic by Lemma \ref{small stuff}, $T_{\alpha}$ contains all elements of order $k$ by Lemma \ref{order els}. Since $I^*_{\alpha}$ and $T_{\alpha}$ are coprime, there exists $r,s\in \mathbb{Z}$ such that $rm+sk=1$, and so 
\[ a_{\alpha}=a_{\alpha}^{rm+sk}=(a_{\alpha}^m)^r(a_{\alpha}^k)^s\in T_{\alpha}I^*_{\alpha}. 
\] 
Now let $b_{\alpha}$ be an element of infinite order. If there exists  $\beta$ such that $a_{\alpha}\psi_{\alpha,\beta}$ has finite order $n$ then $a_{\alpha}^n\in T_{\alpha}$ and so $T_{\alpha}$ contains all elements of infinite order. Otherwise, no such $\beta$ exists, so that $b_{\alpha}\in I^*_{\alpha}$ and $I^*_{\alpha}$ contains all elements of infinite order. Hence $G_{\alpha}=T_{\alpha}I^*_{\alpha}$. 

Now suppose $I^*_{\alpha}$ forms a subgroup. Then as $I^*_{\alpha}$ and $T_{\alpha}$ are trivial intersecting characteristic, and thus normal, subgroups of $G_{\alpha}$ and $G_{\alpha}=T_{\alpha}I^*_{\alpha}$, it follows that $G_{\alpha}=T_{\alpha}\otimes I^*_{\alpha}$. 
If $G_{\alpha}$ is non-periodic, then we have shown that every element of $G_\alpha$ of infinite order lies in $T_{\alpha}\cup I_{\alpha}^*$. It follows by a similar argument  to the proof of Lemma \ref{coprime} that $G_{\alpha}$ equals $T_{\alpha}$ or $I^*_{\alpha}$. 
Since the connecting morphisms are non-injective we thus have $G_{\alpha}=T_{\alpha}$ and $I^*_\alpha$ is trivial. 
\end{proof}

\begin{lemma}\label{prime closure} The subset $I^*_{\alpha}$ is closed under prime powers. Moreover, $I^*_{\alpha}$ forms a subgroup if and only if $\langle I^*_{\alpha} \rangle_I$ and $T_{\alpha}$ intersect trivially.
\end{lemma} 

\begin{proof} Suppose $p\in o(I_{\alpha}^*)$ for some prime $p$, and let $g_{\alpha}\in G_{\alpha}$ be of order $p^r$. Then $g_{\alpha}^{p^{r-1}}$ has order $p$, and is therefore an element of $I_{\alpha}^*$ by Lemma \ref{order els}. Hence $(g_{\alpha}\psi_{\alpha,\beta})^{p^{r-1}}=g_{\alpha}^{p^{r-1}}\psi_{\alpha,\beta}$ is of order $p$ for any $\beta<\alpha$, and $(g_{\alpha}\psi_{\alpha,\beta})^{p^r}=e_\beta$, so that $o(g_{\alpha}\psi_{\alpha,\beta})=p^r=o(g_{\alpha})$. We thus have that $g_{\alpha}\in I_{\alpha}^*$, and so $I_\alpha^*$ is closed under prime powers
 
 Now suppose $\langle I^*_{\alpha} \rangle_I \cap T_{\alpha}=\{e_{\alpha}\}$ and let $g_{\alpha},h_{\alpha}\in I^*_{\alpha}$. If $(g_{\alpha}h_{\alpha}^{-1})\psi_{\alpha,\beta}$ has finite order $m$ for some $\beta<\alpha$ then $(g_{\alpha}h_{\alpha}^{-1})^m\in K_{\alpha,\beta}\subseteq T_{\alpha}$. However, $(g_{\alpha}h_{\alpha}^{-1})^m\in \langle I^*_{\alpha} \rangle_I$ and so $(g_{\alpha}h_{\alpha}^{-1})^m=e_{\alpha}$. It follows that $g_{\alpha}h_{\alpha}^{-1}$ has order $m$, since its order is at least the order of its image, and so $g_{\alpha}h_{\alpha}^{-1}\in I^*_{\alpha}$. Otherwise, $(g_{\alpha}h_{\alpha}^{-1})\psi_{\alpha,\beta}$ has infinite order for all $\beta<\alpha$, so that $g_{\alpha}h_{\alpha}^{-1}$ is also of infinite order, and so $g_{\alpha}h_{\alpha}^{-1}\in I^*_{\alpha}$. The converse is immediate from Lemma \ref{small stuff}.
 \end{proof}  


This lemma points towards a positive answer to the following question. 

\begin{open} Are the absolute images of a surjective Clifford HIS with non-injective connecting morphisms necessarily subgroups? 
\end{open} 

\begin{lemma}\label{ab image trivial} If the absolute images of $S$ form subgroups isomorphic to $I^*$ then
\[ S\cong T(S) \bowtie (Y\times I^*), 
\]
where if $S$ is non-periodic then $I^*$ is trivial. 
\end{lemma}

\begin{proof} 
 Suppose first that $S$ is periodic. Since $I^*_{\alpha}$ is a subgroup, $A(S)=[Y;I^*_{\alpha};\psi_{\alpha,\beta}^I]$ is a surjective Clifford subsemigroup of $S$ by Lemma \ref{subsemigroups}. By Lemma \ref{surj splits} we have $G_{\alpha}=T_{\alpha}\otimes I_{\alpha}^*$, where $T_{\alpha}$ and $I_{\alpha}^*$ are coprime, so that $\psi_{\alpha,\beta}=\psi^T_{\alpha,\beta}\otimes \psi_{\alpha,\beta}^I$ by Lemma \ref{iso of direct prod}. 
 Hence $S\cong T(S) \bowtie I(S)$ by Lemma \ref{spined to cliff}.
   Moreover, we have $I^*_{\alpha}\cap T_{\alpha}=\{e_{\alpha}\}$ by Lemma \ref{small stuff}, so in particular $I^*_{\alpha}\cap \text{Ker }\psi_{\alpha,\beta}^I=\{e_{\alpha}\}$, and so each connecting morphism $\psi_{\alpha,\beta}^I$ is injective, and thus an isomorphism. Hence $[Y;I_{\alpha};\psi_{\alpha,\beta}^I]\cong Y\times I^*$ by Lemma \ref{iso state}, and the first result then follows from Corollary \ref{iso spined} and the fact that $Y\times I^*$ is a structure-HIS by Proposition \ref{isos structure hom}. 

The final result is immediate from Lemma \ref{surj splits}. 
\end{proof} 

We  have proven the forward direction direction of the subsequent Theorem. The converse holds by Proposition \ref{structure sums 2}, since the inverse semigroup $Y\times I^*$ is a structure-HIS when $Y$ and $I^*$ are homogeneous.  

\begin{theorem}\label{surj main result} Let $S$ be a surjective Clifford semigroup such that each absolute image forms a subgroup and the connecting morphisms are surjective but not injective. Then $S$ is a HIS if and only if there exists a homogeneous semilattice $Y$, a homogeneous group $G=T \otimes I^*$ where $T$ and $I^*$ are of coprime order if $G$ is periodic, or  $I^*$ is trivial otherwise, such that $S$ is isomorphic to 
\[ [Y;T_{\alpha};\psi_{\alpha,\beta}^T]\bowtie (Y\times I^*)
\] 
where $T_{\alpha}\cong T$ for each $\alpha\in Y$ and $[Y;T_{\alpha};\psi_{\alpha,\beta}^T]$ is a surjective Clifford HIS with  $T_{\alpha}$ being the union of the kernels, none of which are equal. 
\end{theorem} 

In the case when the absolute images forms subgroups, it consequently suffices to consider the homogeneity of a surjective Clifford semigroup $[Y;T_{\alpha};\psi_{\alpha,\beta}]$, with $Y$ and $T_{\alpha}$ homogeneous and  $T_{\alpha}$ being a (dense) union of the kernels of the connecting morphisms, none of which are equal. 

\begin{open}  Which homogeneous groups are a dense union of isomorphic normal subgroups? 
\end{open} 

A group is \textit{co-Hopfian} if it is not isomorphic to a proper subgroup\footnote{Non co-Hopfian groups are also known as $I$-groups.}. This is equivalent to every injective morphism being an automorphism. Dually, a group is \textit{Hopfian} if it is not isomorphic a proper quotient or, equivalently, if every surjective morphism is an automorphism. An immediate consequence of the following lemma is that $T_{\alpha}$ is both non Hopfian and non co-Hopfian:  

\begin{lemma}\label{T is kernel} Let  $[Z;H_{\alpha};\phi_{\alpha,\beta}]$ be a HIS with each connecting morphism surjective but not injective and such that $H_{\alpha}=\bigcup_{\beta<\alpha}$ Ker $\phi_{\alpha,\beta}$ for each $\alpha\in Z$. Then $H_{\alpha}$ is non hopfian and non co-Hopfian, with $H_{\alpha}\cong$  Ker $\phi_{\alpha,\beta}$.
\end{lemma} 

\begin{proof} For each $\alpha>\beta$, let $K_{\alpha,\beta}=$ Ker $\phi_{\alpha,\beta}$, noting that $K_{\alpha,\beta}$ is homogeneous by Lemma \ref{I char}. We claim that age($K_{\alpha,\beta}$)=age($H_{\alpha}$), so that  $K_{\alpha,\beta}\cong H_{\alpha}$ by Fra\"iss\'e's Theorem.
 Indeed, as $K_{\alpha,\beta}$ is a subgroup of $H_{\alpha}$ we have that age($K_{\alpha,\beta}$) is a subclass of age($H_{\alpha}$). Let $A\in$ age($H_{\alpha}$). Then there exists a f.g. inverse subsemigroup $A'=\langle g_{\alpha,1},\dots, g_{\alpha,n} \rangle_I$ of $H_{\alpha}$ isomorphic to $A$. 
For each $1\leq i \leq n$ there exists $\beta_i<\alpha$ such that $g_{\alpha,i}\in K_{\alpha,\beta_i}$. Letting $\gamma=\beta_1\beta_2\cdots \beta_n$, then $g_{\alpha,i}\in K_{\alpha,\gamma}$ for all $i$ and so $A'\subseteq K_{\alpha,\gamma}$. Hence, as $K_{\alpha,\beta}\cong K_{\alpha,\gamma}$ by Lemma \ref{inj iff inj} we have $A'\in$ age($K_{\alpha,\beta})$. 
Since the age of a structure is closed under isomorphism, we have $A\in$ age($K_{\alpha,\beta})$, thus completing the proof of the claim, and so $H_{\alpha}$ is non co-Hopfian.  

By Corollary \ref{structure clifford} we may pick an isomorphism $\theta:H_{\alpha}\rightarrow H_{\beta}$. The endomorphism of $H_{\alpha}$ given by  $\phi_{\alpha,\beta} \theta^{-1}$ is a surjective non-automorphism, and thus $H_{\alpha}$ is non Hopfian.
\end{proof} 

\section{Homogeneity of commutative inverse semigroups} 

Given that a full classification of homogeneous abelian groups is known, it is natural to consider an extension of this to commutative inverse semigroups. As an immediate consequence of \cite[Theorem 4.2.1]{Howie94}, commutative inverse semigroups are Clifford, and as such we may use the results of the previous sections to attempt to classify commutative HIS. For consistency with earlier work, we continue with the multiplicative notation, so that the operation is denoted by juxtaposition. 

 By Theorem \ref{homog Thm} it suffices to consider the homogeneity of either surjective Clifford semigroups or non-periodic Clifford semigroups with elements of infinite order not lying in the images or absolute-kernels of the maximal subgroups. We first give an overview of homogeneous abelian groups, and consider when such groups are (co-)Hopfian.

Given a prime $p$, the \text{Pr\"ufer} group $\mathbb{Z}[p^{\infty}]$ is an abelian $p$-group with presentation 
\[ \langle g_1,g_2,\cdots : g_1^p=1, g_2^p=g_1,g_3^p=g_2,\cdots \rangle. 
\] 
Alternatively, $\mathbb{Z}[p^{\infty}]$ can be thought of as a union of a chain of cyclic $p$-groups of orders $p,p^2,p^3,\dots$, so that $o(\mathbb{Z}[p^{\infty}])$ is the set of all powers of $p$. 
 Each Pr\"ufer group is divisible, that is, for each $g\in \mathbb{Z}[p^{\infty}]$ and $n\in \mathbb{N}$, there exists $h\in \mathbb{Z}[p^{\infty}]$ such that $h^n=g$. The Pr\"ufer groups, along with $\mathbb{Q}$, form the building blocks for all divisible abelian groups (see \cite{Robinson} for an in depth study of divisible groups).
 
  By \cite[Theorem 2]{Cherlin91}, an abelian group is homogeneous if and only if its isomorphic to some
 \begin{equation} \label{G form}
 G = \begin{cases} (\bigoplus_{p\in P_1} \mathbb{Z}_{p^{m_p}}^{n_p}) \oplus (\bigoplus_{p\in P_2} \mathbb{Z}[p^{\infty}]^{n_p}) &\mbox{if } G \text{ is periodic,} \\
(\bigoplus_{p\in \mathbb{P}} \mathbb{Z}[p^{\infty}]^{n_p}) \oplus (\mathbb{Q}^{n}) & \mbox{otherwise, }  \end{cases} 
\end{equation}  
where $P_1$ and $P_2$ partition the set $\mathbb{P}$ of primes, $n_p,n\in \mathbb{N}^*\cup \{0\}$ and $m_p\in \mathbb{N}$. For example, if $n=\aleph_0=n_p$ for each $p\in \mathbb{P}$ then $G$ is the \textit{universal abelian group}, that is the homogeneous abelian group in which every f.g. abelian group embeds. 

 Note that the groups $\mathbb{Z}_{p^{m_p}}, \mathbb{Z}[p^{\infty}]$ and $\mathbb{Q}$ are \textit{indecomposable}, that is, they are not isomorphic to a direct sum of two non-trivial groups (again see \cite{Robinson}). 

It follows by the work in \cite{Beaumont} that the group $G$ is co-Hopfian if and only if $n$ and $n_p$ are finite, for all $p\in\mathbb{P}$. We shall call $G$ \textit{component-wise non co-Hopfian} if $n,n_p\in \{0,\aleph_0\}$ for each $ p$. That is, $G$ is component-wise non co-Hopfian if and only if each of its non-trivial $p$-components are non co-Hopfian and $n\in \{0,\aleph_0\}$.

Let $H$ be an abelian group with subset $A=\{h_i:i\in I\}$ for some index set $I$. We call $A$ a \textit{disjoint subset} if $\langle A \rangle_I =\bigoplus_{i\in I} \langle h_i \rangle$, or equivalently, if $\langle h_i \rangle_I$ and $\langle A\setminus \{h_i\} \rangle_I$ have trivial intersection for each $i\in I$. Note that if $\{g,h\}$ form a disjoint subset of $H$ then $o(gh)=$ lcm $(o(g),o(h))$, where we define lcm$(\aleph_0,n)=\aleph_0$ for all $n\in \mathbb{N}^*$.  

For example, if $H=Z^{n}$, where $n\in \mathbb{N}^*$ and $Z$ is either a finite cyclic $p$-group, a Pr\"ufer group or $\mathbb{Q}$, then a maximal disjoint subset of $H$ is of size $n$ since $Z$ is indecomposable. 

\subsection{Surjective commutative inverse semigroups} 

Throughout this subsection we let $S=[Y;G_{\alpha};\psi_{\alpha,\beta}]$ be a surjective commutative HIS with each $G_{\alpha}$ isomorphic to the group $G$ in \eqref{G form} and connecting morphisms non-injective. Recall that as $S$ is surjective, each absolute kernel is trivial by Corollary \ref{surj trivial ak}.

\begin{lemma}\label{comm ab} For each $\alpha\in Y$, the absolute image of $G_{\alpha}$ is a subgroup. 
\end{lemma} 

\begin{proof} By Lemma \ref{surj splits} the result is immediate if $G_{\alpha}$ is non-periodic, and so we may assume $G_{\alpha}$ is periodic.
 Let $a_{\alpha},b_{\alpha}\in I_{\alpha}^*$ be of orders $n$ and $m$, respectively. Suppose for contradiction that $a_{\alpha}b_{\alpha}\notin I_{\alpha}^*$, so that there exists $\beta<\alpha$ and $k<o(a_{\alpha}b_{\alpha})$ such that $(a_{\alpha}b_{\alpha})^k\in K_{\alpha,\beta}\subseteq T_{\alpha}$. However, $o((a_{\alpha}b_{\alpha})^k)$ divides $o(a_{\alpha}b_{\alpha})$, which in turn divides $nm$ since $G_{\alpha}$ is abelian. It easily follows that $I_{\alpha}^*$ and $K_{\alpha,\beta}$ are not of coprime orders, contradicting the work after Lemma \ref{small stuff}, and so $I_{\alpha}^*$ is indeed a subgroup. 
\end{proof}


By Theorem \ref{surj main result}, it suffices to consider the case where the absolute image of each maximal subgroup is trivial, so that $G_{\alpha}=\bigcup_{\beta<\alpha}K_{\alpha,\beta}$ for each $\alpha\in Y$.  
Hence, by Lemma \ref{T is kernel}, $G_{\alpha}$ is non Hopfian and non co-Hopfian.
 
\begin{lemma}\label{inf prod sum} The group $G$ is component-wise non co-Hopfian. 
\end{lemma} 

\begin{proof} For each $\alpha\in Y$, let $G_{\alpha}(p)$ denote the $p$-component of $G_{\alpha}$, and let $G_\alpha (p)$ be non-trivial for some $p$. 
 Then $G_{\alpha}(p)$ is an order-characteristic subgroup of $G_{\alpha}$, so that the set $S_p$ of elements of $S$ of order some power of $p$ forms a HIS by Lemma \ref{char sub his}.
  Since $S_p$ is periodic with trivial absolute kernel and absolute image, it follows from Theorem \ref{homog Thm} and Theorem \ref{surj main result} that $S_p$ is a surjective Clifford semigroup with each $G_{\alpha}(p)$ a union of kernel subgroups. In particular, $G_{\alpha}(p)$ is non co-Hopfian, and thus so is the $p$-component of $G$, forcing $n_p=\aleph_0$ by \cite{Beaumont}.

 Suppose for contradiction that $0<n<\aleph_0$, so that $G$ is non-periodic, and let $\alpha>\beta$ in $Y$. Pick a disjoint subset $A=\{g_{\beta,i}:1\leq i \leq n\}$ of $G_{\beta}$ with $o(g_{\beta,i})=\aleph_0$, and let $g_{\alpha,i}\psi_{\alpha,\beta}=g_{\beta,i}$ for each $i$, so that $\{g_{\alpha,i}:1\leq i \leq n\}$ forms a disjoint subset of $G_{\alpha}$. Then for any $x_{\alpha}\in K_{\alpha,\beta}$ of infinite order we have $x_{\alpha}^m=g_{\alpha,1}^{m_1}\cdots g_{\alpha,n}^{m_n}$ for some large enough $m,m_i\in \mathbb{N}$, since otherwise $\{x_{\alpha},g_{\alpha,i}:1\leq i \leq n\}$ forms a disjoint subset of $G_{\alpha}$ of size $n+1$. Hence
\[ e_{\beta}=g_{\beta,1}^{m_1}\cdots g_{\beta,n}^{m_n}
\] 
contradicting $A$ being a disjoint subset. Hence $n$ is infinite as required. 
\end{proof} 

In particular, age$(G)$ is precisely the class of all f.g. abelian groups with elements of order from $o(G)$. Note also that if $G$ is divisible, that is either periodic with $n_p=0$ for each $p\in\mathbb{P}_1$ or non-periodic, then $G$ is a characteristic subgroup of the universal abelian group. 

\begin{lemma}\label{Y uni} The semilattice $Y$ is the universal semilattice.  
\end{lemma} 

\begin{proof} Suppose for contradiction that $Y$ is a linear or semilinear order and let $\alpha>\beta$ in $Y$. Let $g_{\alpha}\in K_{\alpha,\beta}$ be of order $n\in \mathbb{N}^*$. Since the absolute kernel is trivial, there exists $\gamma<\alpha$ such that $g_{\alpha}\not\in K_{\alpha,\gamma}$. Then $\beta\ngtr \gamma$, since otherwise $K_{\alpha,\beta}\subseteq K_{\alpha,\gamma}$. Hence as $Y\cdot \alpha$ forms a chain, we have $\alpha>\gamma>\beta$, so that $K_{\alpha,\gamma}\subsetneq K_{\alpha,\beta}$ by Lemma \ref{kernel neq}. 
Since $K_{\alpha,\beta}\cong K_{\alpha,\gamma}$ by Lemma \ref{inj iff inj} there exists an element $h_{\alpha}\in K_{\alpha,\gamma}$ of order $n$. 

Suppose first that $n=p$ for some prime $p$. 
If there exists $0<k,\ell\leq p$ such that $g_{\alpha}^k=h_{\alpha}^{\ell}$ then 
\[ g_{\alpha}^k\psi_{\alpha,\gamma}=h_{\alpha}^{\ell}\psi_{\alpha,\gamma} = e_{\gamma} 
\] 
and so $g_{\alpha}^k\in K_{\alpha,\gamma}$. However $g_{\alpha}\psi_{\alpha,\gamma}$ has order $p$, and thus $g_{\alpha}^k=e_{\alpha}$, so that $k=p$. Hence $\ell=p$ as $h_{\alpha}$ is of order $p$, and it follows that $\langle g_{\alpha},h_{\alpha}\rangle_I \cong \mathbb{Z}_p\oplus \mathbb{Z}_p$. In particular, we may extend the isomorphism swapping $g_{\alpha}$ and $h_{\alpha}$ to an automorphism $\theta=[\theta_{\alpha},\pi]_{\alpha\in Y}$ of $S$. Since $Y$ is linear or semilinear either $\gamma\geq \gamma\pi$ or $\gamma\pi\geq \gamma$. By the commutativity of the diagram $[\alpha,\gamma;\alpha,\gamma\pi]$  we have 
\[ 
h_{\alpha}\psi_{\alpha,\gamma\pi}=g_{\alpha}\theta_{\alpha}\psi_{\alpha,\gamma\pi} = g_{\alpha}\psi_{\alpha,\gamma}\theta_{\gamma}\neq e_{\gamma\pi}
\] 
as $g_{\alpha}\notin K_{\alpha,\gamma}$. Hence $h_{\alpha}\not\in K_{\alpha,\gamma\pi}$, and similarly as $h_{\alpha}\theta_\alpha=g_{\alpha}$ we attain $g_{\alpha}\in K_{\alpha,\gamma\pi}$. If $\gamma \geq \gamma\pi$ then $h_{\alpha}\in K_{\alpha,\gamma}\subseteq K_{\alpha,\gamma\pi}$, while if $\gamma\pi \geq \gamma$ then $g_{\alpha}\in K_{\alpha,\gamma\pi}\subseteq K_{\alpha,\gamma}$, both giving contradictions. Consequently no element of $G$ can have prime order, and so $G$ is torsion free with $n=\aleph_0$ by Lemma \ref{inf prod sum}. 

 If $g_{\alpha}^k=h_{\alpha}^{\ell}$ for some $k,l\in \mathbb{N}$ then $g_{\alpha}^k\in K_{\alpha,\gamma}$, and so $(g_{\alpha}\psi_{\alpha,\gamma})^k=e_{\gamma}$, contradicting $G$ being torsion free.
  Thus $\langle g_{\alpha},h_{\alpha}\rangle_I \cong \mathbb{Z}\oplus \mathbb{Z}$, and we  argue in much the same way as above to arrive at a contradiction. 
\end{proof} 

 Let $\mathcal{K}(G)$ denote the age of the group $G$, which is  a Fra\"iss\'e class by the homogeneity of $G$. Let $\mathcal{K}[G]$ denote the class of all f.g. commutative inverse semigroups with maximal subgroups from $\mathcal{K}(G)$. That is, $\mathcal{K}[G]$ is the class of all f.g. commutative inverse semigroups with elements of order from $o(G)$. 
 
Given a semigroup $T$, then we may adjoin an extra element $1$ to $T$ to form a monoid $T^1$, with multiplication extending that of $T$ and with $1s=s1=s$ for all $s\in T$ and $11=1$.

\begin{proposition}\label{first fraisse} The class $\mathcal{K}[G]$ forms a Fra\"iss\'e class. 
\end{proposition} 
\begin{proof} Since $\mathcal{K}(G)$  is closed under substructure and direct product, so is $\mathcal{K}[G]$. It is then immediate that $\mathcal{K}[G]$ has HP and JEP. To show closure under amalgamation, we follow the construction by Imaoka in \cite[Section 2]{Imaoka75}. Given $T,T' \in \mathcal{K}[G]$ with common inverse subsemigroup $U$, we may assume w.l.o.g. that $T$ and $T'$ are strong semilattice of groups, say, $T=[Z;H_{\alpha};\phi_{\alpha,\beta}]$ and $T'=[Z';H_{\alpha'}';\phi_{\alpha',\beta'}']$. We may also assume w.l.o.g. that $T\cap T'=U$. Note that $T^1$ has maximal subgroups $\{1\}$ and $H_{\alpha}$ ($\alpha\in Z$), which are members of $\mathcal{K}(G)$, and so $T^1$ is a member of $\mathcal{K}[G]$, and similarly so is $T'^1$. Hence $W=T^1\times  T'^1\setminus \{(1,1)\}$, as an inverse subsemigroup of $T^1\times T'^1$ is a member of $\mathcal{K}[G]$, and is isomorphic to the   free  commutative  inverse  semigroup  product of $T$ and $T'$. Imaoka then showed that there exists a congruence $\rho$ on $W$ such that $W/\rho$ is isomorphic to the free product of $T$ and $T'$ amalgamating $U$ in the variety of commutative inverse semigroups, denoted $T *_U T'$. Moreover, the maps $\theta:T\rightarrow W/\rho$ and $\theta':T'\rightarrow W/\rho$ given by $x\theta=(x,1)\rho$ and $x'\theta'=(1,x')\rho$ ($x\in T,x'\in T'$) are embeddings such that $U\theta=U\theta'=T\theta\cap T'\theta'$. Hence $W/\rho$ is generated by the elements $x\theta$ and $x\theta'$, which are of orders from $o(T)\cup o(T')\subseteq o(G)$, and so since $o(G)$ is closed under product we have $o(W/\rho)\subseteq o(G)$. Consequently $W/\rho$ is a member of $\mathcal{K}[G]$, and AP holds.

Finally, it is somewhat folklore that the class $\mathcal{K}$ of all f.g. commutative semigroups is countable. In fact this result can be deduced from R\'edei's Theorem \cite{Redei}, which states that every f.g. commutative semigroup is finitely presented (see also \cite[Theorem 9.28]{Clif&Pres67}). In particular the subclass $\mathcal{K}[G]$ of $\mathcal{K}$ is countable.  
\end{proof}

We shall denote the Fra\"iss\'e limit of $\mathcal{K}[G]$ as $\mathcal{C}[G]$, noting that $\mathcal{C}[G]\cong \mathcal{C}[G']$ if and only if $G\cong G'$. 
We now prove that $\mathcal{C}[G]$ is isomorphic to $S$. 

    
\begin{lemma}\label{bad facts} Let $m,n\in o(G_{\alpha})$ be such that either $m|n$ or $n=\aleph_0$. Then:
\begin{enumerate}[label=(\roman*), font=\normalfont]
\item If $\alpha>\beta$ then for every $x_{\beta}\in G_{\beta}$ of order $m$  there exists an infinite  disjoint subset of $G_{\alpha}$ of elements of order $n$ which are the pre-image of $x_{\beta}$ under $ \psi_{\alpha,\beta}$;
\item If $\alpha>\{\beta,\gamma\}>\tau$ forms a diamond in $Y$, then for any $x_{\gamma}\in K_{\gamma,\tau}$ of order $m$,  there exists $x_{\alpha}\in K_{\alpha,\beta}$ of order $n$ such that $x_{\alpha}\psi_{\alpha,\gamma}=x_{\gamma}$.
\end{enumerate}
\end{lemma} 

\begin{proof}
$(\mathrm{i})$ Let $\alpha>\beta$ and $x_{\beta}\in G_{\beta}$ be of order $m$. We first claim that there exists $x_{\alpha}\in G_{\alpha}$ of order $m$ with $x_{\alpha}\psi_{\alpha,\beta}=x_{\beta}$. If $m=\aleph_0$ then the result is immediate as $\psi_{\alpha,\beta}$ is surjective. Let $m=p^r$ for some prime $p$ and $r>0$ (the case $r=0$ being trivial). Suppose for contradiction that for all $\gamma<\beta$ we have $o(x_{\beta}\psi_{\beta,\gamma})<p^r$. Then $x_{\beta}^{p^{r-1}}\in K_{\beta,\gamma}$ for all $\gamma<\beta$, contradicting the absolute kernel being trivial. Hence there exists $x_{\gamma}=x_{\beta}\psi_{\beta,\gamma}$ of order $p^r$. By Lemma \ref{iso bit} we may extend the isomorphism $\phi$ from $\{e_{\beta}\}\cup \langle x_{\gamma} \rangle_I$ to $\{e_{\alpha}\}\cup \langle x_{\beta} \rangle_I$, determined by $e_{\beta}\phi =e_{\alpha}$  and $x_{\gamma}\phi=x_{\beta}$,  to an automorphism $[\theta_{\alpha},\pi]_{\alpha\in Y}$ of $S$.
 Then $[\beta,\gamma;\alpha,\beta]$ commutes and so 
\[ x_{\beta}\theta_{\beta}\psi_{\alpha,\beta}=x_{\beta}\psi_{\beta,\gamma}\theta_{\gamma}=x_{\gamma}\theta_{\gamma}=x_{\beta}.
\] 
Since $x_\beta \theta_\beta\in G_{\alpha}$ has order $p^r$, the claim holds in this case. Now suppose $m=p_1^{r_1}p_2^{r_2}\cdots p_s^{r_s}$ for some primes $p_i$ and $r_i\in \mathbb{N}$. By the Fundamental Theorem of Finite Abelian Groups, $x_\beta=x_{\beta,1}x_{\beta,2}\cdots x_{\beta,s}$ for some $x_{\beta,i}\in G_{\beta}$ of order $p_i^{r_i}$ and so, by the previous case, there exists $x_{\alpha,i}\in G_{\alpha}$ of order $p_i^{r_i}$ with $x_{\alpha,i}\psi_{\alpha,\beta}=x_{\beta,i}$ for each $i$. Then $x_{\alpha}=x_{\alpha,1}x_{\alpha,2}\cdots x_{\alpha,s}$ has order $m$ and is such that $x_{\alpha}\psi_{\alpha,\beta} = x_{\beta}$, and the claim holds in all cases.  

Notice that the set $x_{\alpha}K_{\alpha,\beta}$ is precisely the elements of $G_{\alpha}$ mapped to $x_{\beta}$. Since $K_{\alpha,\beta}\cong G_{\alpha}$ by Lemma \ref{T is kernel}, $K_{\alpha,\beta}$ is component-wise non co-Hopfian. 
 By Lemma \ref{inf prod sum} there exists an infinite disjoint subset $\{g_{\alpha,i}: i\in \mathbb{N}\}$ of $K_{\alpha,\beta}$ of elements of order $n$.  If $g_{\alpha,i}^k=x_{\alpha}^l$ for some $0<k<n$ and $0<l<m$ then $e_{\beta}=e_{\beta}^k=x_{\beta}^l$, a contradiction. It follows that each $x_{\alpha}g_{\alpha,i}$ has order $n$. We claim that $\{x_{\alpha}g_{\alpha,i}:i\in \mathbb{N}\}$ forms an infinite disjoint subset of $G_{\alpha}$. If 
\[ (x_{\alpha}g_{\alpha,i})^k=(x_{\alpha}g_{\alpha,j_1})^{k_1}\cdots (x_{\alpha}g_{\alpha,j_t})^{k_t}
\] 
for some $0<k,k_1,\dots,k_t<n$, then $g_{\alpha,i}^k g_{\alpha,j_1}^{-k_1}\cdots g_{\alpha,j_t}^{-k_t}= x_{\alpha}^{k_1k_2\cdots k_t -k}$ by commutativity, and so $e_{\beta}= x_{\beta}^{k_1k_2\cdots k_t -k}$. Hence $m|(k_1k_2\cdots k_t -k)$, so that $x_{\alpha}^{k_1k_2\cdots k_t-k}=e_{\alpha}$, and we thus have 
\[ g_{\alpha,i}^k=g_{\alpha,j_1}^{k_1}\dots g_{\alpha,j_t}^{k_t}
\] 
contradicting $\{g_{\alpha,i}:i\in \mathbb{N}\}$ being disjoint, and the claim holds. The result then follows as $x_{\alpha}g_{\alpha,i}\in x_{\alpha}K_{\alpha,\beta}$. 

$(\mathrm{ii})$ Let  $\alpha>\{\beta,\gamma\}>\tau$ form a diamond in $Y$ and $x_{\gamma}\in K_{\gamma,\tau}$ be of order $m$. By part $(\mathrm{i})$ there exists $y_{\alpha}\in G_{\alpha}$ of order $n$ such that $y_{\alpha}\psi_{\alpha,\gamma}=x_{\gamma}$, so that $y_{\alpha}\in K_{\alpha,\tau}$. Let $y_{\alpha}\psi_{\alpha,\beta}=y_{\beta}$, so $y_{\beta}\in K_{\beta,\tau}$. Then there exists $\beta'\in Y$ with $\beta>\beta'>\tau$ and such that $y_{\beta}\in K_{\beta,\beta'}$, else $\tau$ is a maximal element in the subsemilattice $\{\rho\in Y:y_{\beta}\in K_{\beta,\rho}\}$ of $Y$, which would clearly contradict homogeneity. Extend the isomorphism between $\{e_{\alpha} \} \cup \{ e_{\beta'}\} \cup \langle x_{\gamma} \rangle_I \cup \{e_{\tau}\}$ and $\{e_{\alpha} \} \cup \{ e_{\beta}\} \cup \langle x_{\gamma} \rangle_I \cup \{e_{\tau}\}$ which sends $e_{\beta'}$ to $e_{\beta}$ and fixes all other elements, to an automorphism $\theta'=[\theta'_{\alpha},\pi']_{\alpha\in Y}$ of $S$.  Then $y_{\alpha}\theta'\in G_{\alpha}$ is of order $n$ with $y_{\alpha}\theta'\in K_{\alpha,\beta}$  by the commutativity of $[\alpha,\beta';\alpha,\beta]$ (as $y_{\alpha}\in K_{\alpha,\beta'}$),  and $y_{\alpha}\theta' \psi_{\alpha,\gamma}=x_{\gamma}$ by the commutativity of $[\alpha,\gamma;\alpha,\gamma]$. 
%
\end{proof}

\begin{lemma} Let $\alpha>\{\beta,\gamma\}>\tau$ be a diamond in $Y$ and let $x_{\beta}\in G_{\beta}$ and $x_{\gamma}\in G_{\gamma}$ be of orders $m_1,m_2\in \mathbb{N}^*$, respectively, such that $x_{\beta}\psi_{\beta,\tau}=x_{\tau}=x_\gamma\psi_{\gamma,\tau}$. Then, for any $n\in o(G_{\alpha})$ such that either $m_i|n$ ($i=1,2$) or $n=\aleph_0$, there exists $x_{\alpha}\in G_{\alpha}$ of order $n$ such that $x_{\alpha}\psi_{\alpha,\beta}=x_{\beta}$ and $x_{\alpha} \psi_{\alpha,\gamma} = x_{\gamma}$.
\end{lemma} 

\begin{proof} We may assume that $m_1=m_2=n$. Indeed, as $Y$ is the universal semilattice we may pick $\beta',\gamma' \in Y$ with $\alpha>\beta'> \beta$, $\alpha>\gamma'>\gamma$ and $\beta'\gamma'=\tau$ by Lemma \ref{U fact}. By the previous lemma there exists $x_{\beta'}\in G_{\beta'}$ and $x_{\gamma'}\in G_{\gamma'}$ of order $n$ with $x_{\beta'}\psi_{\beta',\beta}=x_{\beta}$ and $x_{\gamma'}\psi_{\gamma',\gamma}=x_{\gamma}$, and so it would suffice to consider $x_{\beta'}$ and $x_{\gamma'}$ instead. 

By Lemma \ref{bad facts} $(\mathrm{i})$ there exists $z_{\alpha}\in G_{\alpha}$ of order $n$ such that $z_\alpha \psi_{\alpha,\beta}=x_{\beta}$, so that $z_{\alpha}\psi_{\alpha,\tau}=x_{\tau}$. Let $z_\alpha \psi_{\alpha,\gamma}=z_{\gamma}$. Then there exists $\gamma'$ such that $\gamma>\gamma'>\tau$ and $o(z_{\gamma}\psi_{\gamma,\gamma'})=o(x_{\tau})$, else $\tau$ would be a maximal element in the set $\{\rho:o(z_\gamma\psi_{\gamma,\rho})=o(x_{\tau})\}$, thus contradicting the homogeneity of $S$. Let $z_{\gamma}\psi_{\gamma,\gamma'}=z_{\gamma'}$, and pick $g_{\gamma'}\in K_{\gamma',\tau}$ of order $n$, noting that such an element exists as $o(G)=o(K_{\gamma',\tau})$. Arguing in much the same way as in the proof of Lemma \ref{bad facts} $(\mathrm{i})$, the element $z_{\gamma'}g_{\gamma'}$ has order $n$. By Lemma \ref{bad facts} $(\mathrm{ii})$ there exists $g_{\alpha}\in K_{\alpha,\beta}$ of order $n$ with $g_{\alpha}\psi_{\alpha,\gamma'}=g_{\gamma'}$.
Then as $(z_{\alpha}g_{\alpha})\psi_{\alpha,\beta}= x_{\beta}$ has order $n$, it easily follows that $z_{\alpha}g_{\alpha}$ has order $n$, and is such that 
 \[ (z_{\alpha}g_{\alpha})\psi_{\alpha,\beta} =  z_{\alpha}\psi_{\alpha,\beta} = x_{\beta} \quad \text{and} \quad (z_{\alpha}g_{\alpha})\psi_{\alpha,\gamma'} = z_{\gamma'}g_{\gamma'}. 
 \] 
The map between the f.g. inverse subsemigroups $\{e_{\alpha}\}\cup \langle x_{\beta} \rangle_I \cup \langle  z_{\gamma'}g_{\gamma'} \rangle_I \cup \langle x_{\tau} \rangle_I$ and $\{e_{\alpha}\}\cup \langle x_{\beta} \rangle_I \cup \langle  x_{\gamma} \rangle_I \cup \langle x_{\tau} \rangle_I$ which sends $z_{\gamma'}g_{\gamma'}$ to $x_{\gamma}$ and fixes all other elements, is clearly an isomorphism.
 Extend the isomorphism to an automorphism $\theta$ of $S$. Then $(z_{\alpha}g_{\alpha})\theta\in G_{\alpha}$ gives the required element.
\end{proof}

\begin{corollary}\label{bad n} Let $\beta_1,\beta_2,\dots,\beta_r\in Y$ be such that $\beta_i \perp \beta_j$ for each $i\neq j$, and $x_{\beta_i}\in G_{\beta_i}$  be such that if $\gamma<\beta_i,\beta_j$ then $x_{\beta_i}\psi_{\beta_i,\gamma} = x_{\beta_j}\psi_{\beta_j,\gamma}$. Then for any $\alpha\in Y$ with $\alpha>\beta_i$ for all $i$, and any $n\in o(G_{\alpha})$ with either $o(x_{\beta_i})|n$ for all $i$ or $n=\aleph_0$, there exists an infinite disjoint subset of $G_{\alpha}$ of elements of order $n$ which are the pre-image of $x_{\beta_i}$ under $\psi_{\alpha,\beta_i}$.
\end{corollary} 

\begin{proof}
By Lemma \ref{bad facts} ($\mathrm{i}$) the result holds when $r=1$. We proceed by induction by supposing that the result holds when $r=k-1$, and letting $x_{\beta_1},x_{\beta_2}, \dots, x_{\beta_{k}}$ and $n\in o(G_{\alpha})$ satisfy the conditions of the corollary.
Since $Y$ is the universal semilattice there exists $\alpha'\in Y$ with $\alpha>\alpha'>\beta_2,\dots, \beta_k$ but $\alpha'\ngeq \beta_1$ by Lemma \ref{U fact}.
 By the induction hypothesis there exists $x_{\alpha'}$ of order $n$ such that $x_{\alpha'}\psi_{\alpha',\beta_i}=x_{\beta_i}$ for all $2\leq i \leq k$. Since $\alpha>\{\alpha',\beta_1\}>\alpha'\beta_1$ forms a diamond in $Y$, there exists $x_{\alpha}\in G_{\alpha}$ of order $n$ such that $x_{\alpha}\psi_{\alpha,\alpha'}=x_{\alpha'}$ and $x_{\alpha}\psi_{\alpha,\beta_1}=x_{\beta_1}$ by the previous lemma. Hence $x_{\alpha}\psi_{\alpha,\beta_i}=x_{\beta_i}$ for all $i$. Let $\delta\in Y$ be such that $\alpha>\delta>\beta_i$ for each $i$ (again such an element exists by Lemma \ref{U fact}), and let $x_{\alpha}\psi_{\alpha,\delta}=x_{\delta}$. By Lemma \ref{bad facts} $(\mathrm{i})$ there exists an infinite disjoint subset of $G_{\alpha}$ of elements of order $n$ which are mapped to $x_{\delta}$, and thus to each $x_{\beta_i}$.  
This completes the inductive step. 
\end{proof}

\begin{proposition}\label{age 1} The age of $S$ is $\mathcal{K}[G]$. 
\end{proposition} 
\begin{proof} Let $\mathcal{K}$ denote the age of $S$, noting that clearly $\mathcal{K}$ is a subclass of $\mathcal{K}[G]$. Since a 1-generated Clifford semigroup is a cyclic group, each 1-generated member of $\mathcal{K}[G]$ is a member of $\mathcal{K}(G)$, and thus of $\mathcal{K}$. Proceeding by induction, assume that every $n$-generated member of $\mathcal{K}[G]$ is contained in $\mathcal{K}$, for some $n\in \mathbb{N}$. Let $A=[Z;A_{\alpha};\phi_{\alpha,\beta}]$ be an $n+1$-generated member of $\mathcal{K}[G]$. To avoid $A$ trivially being a member of $\mathcal{K}(G)$ we may assume that $Z$ is non-trivial. Let $\alpha$ be maximal in $Z$ and suppose $A_{\alpha}=\langle a_{\alpha,1} \rangle_I \oplus \langle a_{\alpha,2} \rangle_I \oplus \cdots \oplus \langle a_{\alpha,r} \rangle_I$ is an $r$-generated abelian group, where each $\langle a_{\alpha,i} \rangle_I$ is a cyclic subgroup. Let $A'$ be the inverse subsemigroup of $A$ given by 
\begin{align*}
 A' =
  \begin{cases}
    A\setminus A_{\alpha} & \text{if } r=1,\\ 
   (A \setminus A_{\alpha}) \cup \langle a_{\alpha,2} \rangle_I \oplus \langle a_{\alpha,3} \rangle_I \oplus \cdots \oplus  \langle a_{\alpha,r} \rangle_I & \text{if } r>1.   \end{cases} 
\end{align*}
Then $A'$ is $n$-generated and with structure semilattice $\bar{Z}$, where $\bar{Z}=Z\setminus \{\alpha\}$ if $r=1$, and $\bar{Z}=Z$ else. By the inductive hypothesis there exists an embedding $\theta:A'\rightarrow S$, with induced embedding $\pi:\bar{Z} \rightarrow Y$. Since $Y$ is the universal semilattice there exists $\delta\in Y$ such that $\bar{Z}\pi\cup \{\delta\}\cong Z$ by Lemma \ref{U fact}, where we take $\delta=\alpha\pi$ if $r>1$. Let $\alpha$ be an upper cover of $\beta_1,\dots,\beta_r$ in $Z$, and $a_{\alpha,1}\phi_{\alpha,\beta_i}=a_{\beta_i}$ for each $i$. 
By Corollary \ref{bad n}, there exists a infinite disjoint subset $\{g_{\delta,k}:k\in \mathbb{N}\}$ of $G_{\delta}$ with $o(g_{\delta,k})=o(a_{\alpha,1})$ which are the pre-image of each $a_{\beta_i}\theta$ under $\psi_{\delta,\beta_i\pi}$ ($1\leq i \leq r$). Note that $A'_{\alpha}$ is f.g. since it is either empty or equal to $\langle a_{\alpha,2} \rangle_I \oplus \langle a_{\alpha,3} \rangle_I \oplus \cdots \oplus \langle a_{\alpha,r} \rangle_I$. On the other hand, $\bigoplus_{k\in \mathbb{N}} \langle g_{\delta,k} \rangle_I$ is infinitely generated, and it follows that there exists only finitely many $g_{\delta,k}$ with $\langle g_{\delta,k} \rangle_I \cap A_{\alpha}' \theta \neq \{e_{\alpha}\}$. Hence, for some $k\in \mathbb{N}$, we have that $\langle g_{\delta,k} \rangle_I \oplus A_{\alpha}'\theta$ is isomorphic to $A_{\alpha}$, and its easily shown that the map $\theta': A \rightarrow A'\theta \cup \langle g_{\delta,k}\rangle_I$ given by $A'\theta'=A'\theta$ and $a_{\alpha,1}\theta'=g_{\delta,k}$ is an embedding, thus completing the inductive step.
\end{proof}

A full classification of surjective commutative HIS is now achieved. In particular we may describe all periodic commutative HIS as follows (the non-periodic case will be considered separately in the next subsection). 

\begin{theorem}\label{Periodic com} Let $I^*,K^*$ and $T$ be periodic homogeneous abelian groups of pairwise coprime orders and $T$ component-wise non co-Hopfian. Let $Y$ be a homogeneous semilattice, and let $U$ denote the universal semilattice. Then the following inverse semigroups are HIS:  
\begin{enumerate}[label=(\roman*), font=\normalfont]
\item  $(Y \times I^*) \bowtie [Y;K^*]$; 
\item $(U \times I^*) \bowtie \mathcal{C}[T] \bowtie [U;K^*]$; 
\end{enumerate} 
Conversely, every periodic commutative HIS is isomorphic to an inverse semigroup constructed in this way. 
\end{theorem} 

\begin{proof} Let $S$ be a periodic commutative HIS. Then by Theorem \ref{homog Thm} $S\cong I(S) \bowtie [Y;K^*]$, where $I(S)=[Y;I_{\alpha};\psi_{\alpha,\beta}^I]$ is a surjective Clifford HIS and $I_{\alpha}\cong I$ is coprime to the homogeneous group $K^*$. 
By Corollary \ref{surj trivial ak} the absolute kernels of $I(S)$ are trivial.
 If each $\psi_{\alpha,\beta}^I$ is an isomorphism, then $I(S)\cong Y\times I^*$ by Lemma \ref{iso state}, which is structure-HIS by Proposition \ref{isos structure hom}. 
 We then have case $(\mathrm{i})$ by Corollary \ref{iso spined}.
  Otherwise, as the absolute images form subgroups by Lemma \ref{comm ab} we have $I(S)\cong [Y;T_{\alpha};\psi_{\alpha,\beta}^T]\bowtie (Y \times I^*)$ by Theorem \ref{surj main result}, where $T_{\alpha}$ has trivial absolute image and is of coprime order to $I_{\alpha}^*$.
   Each $T_{\alpha}$ is isomorphic to some component-wise non co-Hopfian group $T$ by Lemma \ref{inf prod sum}. By Propositions \ref{first fraisse} and \ref{age 1} we have $[Y;T_{\alpha};\psi_{\alpha,\beta}^T]\cong \mathcal{C}[T]$, and we thus obtain case ($\mathrm{ii}$) again by Corollary \ref{iso spined}.

Conversely, the inverse semigroups $Y\times I^*$ and $[Y;K^*]$ are structure-HIS by Proposition \ref{isos structure hom} and Lemma \ref{trivial mor}, and the semigroup $\mathcal{C}[T]$ is a HIS by Proposition \ref{first fraisse}. 
The result then follows by Proposition \ref{structure sums 2}. 
\end{proof}

\subsection{An open case} 

We now consider the final case, where $S=[Y;G_{\alpha};\psi_{\alpha,\beta}]$  is a commutative HIS such that each $G_{\alpha}$ is isomorphic to the group $G$ in \eqref{G form} and elements of infinite order are not contained in the image $I_{\alpha}$ or absolute kernel $K^*_{\alpha}$ of $G_{\alpha}$. We observe that as $G_{\alpha}\neq K^*_{\alpha}$, the subgroup $I_{\alpha}$ is non-trivial. 

By Lemma \ref{absolute image trivial} the absolute images of $S$ are trivial, so that $G_{\alpha}=\bigcup_{\beta<\alpha}K_{\alpha,\beta}$ and $G_{\alpha}\cong K_{\alpha,\beta}$ by Lemma \ref{T is kernel}.
 By Lemma \ref{coprime}, the elements of $G_{\alpha}$ of finite order form precisely the subgroup $I_{\alpha} \oplus K_{\alpha}^*$, which is clearly a characteristic subgroup of $G_\alpha$. 
 It follows by Lemma \ref{char sub his} that elements of $S$ of finite order forms a HIS 
\[ T= [Y;I_{\alpha}\oplus K^*_{\alpha};\psi_{\alpha,\beta}^I\bowtie \psi_{\alpha,\beta}^K]
\]
where  $\psi_{\alpha,\beta}^I\bowtie \psi_{\alpha,\beta}^K = \psi_{\alpha,\beta}|_{I_{\alpha}\oplus K^*_{\alpha}}$. The inverse subsemigroup $[Y;I_{\alpha};\psi_{\alpha,\beta}^I]$ of $T$ is a periodic  surjective commutative HIS with trivial absolute-images and, by Corollary \ref{surj trivial ak}, trivial absolute kernels.
 It then follows by Theorem \ref{Periodic com} that $[Y;I_{\alpha};\psi_{\alpha,\beta}^I]\cong \mathcal{C}[I]$, where $I\cong I_{\alpha}$ is component-wise non co-Hopfian, and that $Y$ is isomorphic to universal semilattice $U$ by Lemma \ref{Y uni}. 
 Hence $[Y;K^*_{\alpha};\psi_{\alpha,\beta}^K]$ is isomorphic to the structure-HIS $[U;K^*]$ by Lemmas  \ref{trivial mor} and \ref{trivial iso}, where  $K^*\cong K_{\alpha}^*$.   By Corollary \ref{iso spined}, we have that  $T\cong \mathcal{C}[I] \bowtie [U;K^*]$, and it follows that 
\begin{equation} \label{G part 2} G=I \oplus K^* \oplus \mathbb{Q}^n
\end{equation} 
where $I=\bigoplus_{p\in \mathbb{P}_I} \mathbb{Z}[p^{\infty}]^{\aleph_0}$ and $K^*=\bigoplus_{p\in \mathbb{P}_K}\mathbb{Z}[p^{\infty}]^{n_p}$ for some $n,n_p\in \mathbb{N}^*$, where $\mathbb{P}_I$ and $\mathbb{P}_K$ are disjoint subsets of $\mathbb{P}$.


We let $\mathcal{K}^*[I;K;n]$ denote the class of all f.g. commutative inverse semigroups $A$ with maximal subgroups in $\mathcal{K}(G)$, where $G$ is as in \eqref{G part 2}, and satisfying the following properties: 
\begin{enumerate} \item every element of infinite order is maximal in $(A,\leq)$;
\item  for each $p\in \mathbb{P}_K$, every element of order some power of $p$ is maximal in $(T,\leq)$; 
\end{enumerate} 
where $\leq$ is the natural order on $A$. In particular, if $[Z;A_{\alpha};\phi_{\alpha,\beta}]\in \mathcal{K}^*[I;K;n]$ then every element of infinite order is mapped to an element of finite order by non-trivial connecting morphisms by (1) and, for each $p\in \mathbb{P}_K$, every non-trivial element of order some power of $p$ is not contained in an image of any connecting morphism by (2), and so is contained in the absolute kernel of its maximal subgroup. Note also that $\mathcal{K}[I]$ is a subclass of $\mathcal{K}^*[I;K;n]$. 

\begin{open} For which conditions on $K$ and $n$ does $\mathcal{K}^*[I;K;n]$ form a  Fra\"iss\'e class? 
\end{open} 

The problem we face when tackling this is open problem that the method in the proof of Proposition \ref{first fraisse} no longer applies. For example, let $\alpha>\beta$ and $\gamma>\beta$ be a pair of chains, and let $T=\langle x_{\alpha} \rangle_I \cup \{e_{\beta}\}$ and $T'=\{e_{\gamma}\}\cup \{e_{\beta}\}$ be f.g. Clifford semigroups, with $x_{\alpha}$ of infinite order.  Note that $T\cap T'=\{e_{\beta}\}$. Let  $\rho$ be the congruence on $W=T^1\times T'^1\setminus \{(1,1)\}$ as given by Imaoka. Then $(x_{\alpha},1)\rho$ and $(x_{\alpha},e_{\beta}))\rho$ have infinite order and $(x_{\alpha},1)\rho > (x_{\alpha},e_{\beta})\rho$. Hence $W/\rho$ does not satisfy (2), and is thus not a member of $\mathcal{K}^*[I;K;n]$. 

We shall now prove that age($S$) is $\mathcal{K}^*[I;K;n]$ by following the methods of the previous subsection.   

\begin{proposition} Let $\beta_1,\beta_2,\dots,\beta_r\in Y$ be such that $\beta_i \perp \beta_j$ for each $i\neq j$. Let $x_{\beta_i}\in I_{\beta_i}$  be such that if $\gamma<\beta_i,\beta_j$ then $x_{\beta_i}\psi_{\beta_i,\gamma} = x_{\beta_j}\psi_{\beta_j,\gamma}$.
 Then for any $\alpha\in Y$ with $\alpha>\beta_i$ for all $i$, and any $t\in o(G_{\alpha})$ with either $o(x_{\beta_i})|t$ for all $i$ or $t=\aleph_0$,  there exists a disjoint subset of $G_{\alpha}$ of size $\aleph_0$ if $t$ is finite, and size $n$ otherwise, consisting of elements of order $t$ which are the pre-image of $x_{\beta_i}$ under $\psi_{\alpha,\beta_i}$.
\end{proposition} 

\begin{proof}
Recall that $S$ contains the inverse subsemigroup $[Y;I_{\alpha};\psi_{\alpha,\beta}^I]$ isomorphic to $\mathcal{C}[I]$, where $I\cong I_{\alpha}$.  Hence if $t$ is finite, then the result easily follows by Corollary \ref{bad n}, where the required disjoint subset of $G_{\alpha}$ is contained in $I_{\alpha}$. 
 
  Suppose instead that $t=\aleph_0$. By the previous case we may fix $x_{\alpha}\in I_{\alpha}$ of finite order with $x_{\alpha}\psi_{\alpha,\beta_i}=x_{\beta_i}$ for all $i$. Since $Y$ is the universal semilattice, we may pick $\alpha'\in Y$ with $\alpha>\alpha'>\beta_1,\dots, \beta_r$ by Lemma \ref{U fact}.  Let $\{z_i:i\in \mathbb{N}\}$ be a disjoint set of size $n$ consisting of elements of infinite order and such that $z_i\in K_{\alpha,\alpha'}\subseteq K_{\alpha,\beta_i}$. Note that such a set exists as $K_{\alpha,\alpha'}\cong G_{\alpha}$. Then $x_{\alpha}z_i$ is  a disjoint set of size $n$ consisting of elements of infinite order, and $ (x_{\alpha}z_i)\psi_{\alpha,\beta_i}=x_{\alpha}\psi_{\alpha,\beta_i}=x_{\beta_i}$    for each $i\in \mathbb{N}$.   
\end{proof}

By a simple adaptation of Proposition \ref{age 1} we thus obtain: 

\begin{corollary}\label{age 2} The age of $S$ is  $\mathcal{K}^*[I;K;n]$. 
\end{corollary} 

\begin{theorem}[Classification Theorem of commutative HIS] Let $G$ be an homogeneous non-periodic abelian group, $Y$ a homogeneous semilattice and $U$ be the universal semilattice. Then the following inverse semigroups are HIS:  
\begin{enumerate}[label=(\roman*), font=\normalfont]
\item $Y\times G$; 
\item $[Y;G]$; 
\item $\mathcal{C}[G]$, with $G$ component-wise non co-Hopfian; 
\item the Fra\"iss\'e limit of a Fra\"iss\'e class $\mathcal{K}^*[I;K^*;n]$, where $G=I\oplus K^*\oplus \mathbb{Q}^n$ and $I$ is component-wise non co-Hopfian. 
\end{enumerate} 
Conversely, every non-periodic commutative HIS is isomorphic to an inverse semigroup constructed in this way.  
\end{theorem} 

\begin{proof} By Theorem \ref{homog Thm}, A non-periodic commutative Clifford semigroup $S$ is a HIS if and only if isomorphic to either a surjective Clifford HIS, or $[Y;G]$ for some homogeneous semilattice $Y$ and group $G$, or a Clifford HIS with no elements of infinite order lying in the images or absolute kernels.
 By the usual argument, a surjective commutative Clifford is a HIS if and only if is isomorphic to either $Y\times G$ for some homogeneous $Y$ and $G$, or a HIS with connecting morphisms non-injective, and $G_{\alpha}$ being a dense union of kernels by Theorem \ref{surj main result} (since the absolute image forms a subgroup by Lemma \ref{comm ab}).
  By Propositions \ref{first fraisse} and \ref{age 1} the second possibility holds if and only if $S$ isomorphic to $\mathcal{C}[G]$, where $G$ is component-wise co-Hopfian by Lemma \ref{inf prod sum}. 
  It thus suffices to consider the case where $S$ has no elements of infinite order lying in the images or absolute kernels. If $S$ is a HIS then age($S$) is $\mathcal{K}^*[I;K;n]$ by Corollary \ref{age 2}, and thus $S$ is isomorphic to the Fra\"iss\'e limit of the Fra\"iss\'e class $\mathcal{K}^*[I;K;n]$. The converse is trivial by Fra\"iss\'e's Theorem. 
\end{proof}
 
\section{Inverse Homogeneous Semigroups} 

In this section we study the differences between the two concepts of homogeneity for inverse semigroups, in particular when a HIS has the stronger property of being an inverse homogenous semigroup (HS).  

\begin{lemma}  Let $S$ be a periodic inverse semigroup. Then $S$ is a HS if and only if it is a HIS. 
\end{lemma} 

\begin{proof} Suppose $S$ is a HIS, so that, being periodic, $S=[Y;G_{\alpha};\psi_{\alpha,\beta}]$ is a Clifford semigroup by Theorem \ref{cliff of bisimple}. Note that a subsemigroup of a periodic group is a subgroup, since some power of each element is its inverse. Hence, for any finitely generated subsemigroup $A=\langle g_{\alpha_1},\dots,g_{\alpha_n} \rangle$ of $S$, we have $A=\langle  g_{\alpha_1},\dots,g_{\alpha_n} \rangle_I$, and the result follows. The converse is trivial.  
\end{proof}

Now consider an arbitrary (not necessarily inverse) semigroup $S$, and let 
\[ \text{Inf}(S)=\{a\in S:|\langle a \rangle|=\aleph_0\}.
\] 
We observe that if $S$ is a HS then Aut($S$) acts transitively on $\text{Inf}(S)$. Indeed, for each $a,b\in \text{Inf}(S)$, we have 
 \[ \langle a \rangle \cong \mathbb{N} \cong \langle b \rangle,
 \] 
 and the result then follows by the homogeneity of $S$. In particular, either all elements of $\text{Inf}(S)$ lie in subgroups of $S$, or none of them do. Indeed, if $a,b\in \text{Inf}(S)$ are such that $a\in H_e$ for some $e\in E(S)$ then by taking an automorphism of $S$ sending $a$ to $b$ we have $b\in H_{e\theta}$ since $\mathcal{H}$ is preserved under automorphisms (noting that $e\theta\in E(S)$). 
 
 The set of idempotents $E(S)$ of an arbitrary semigroup $S$ forms a partially ordered set under the natural order, given by $f\leq e$ if and only if $ef=f=fe$ (so that if $S$ is inverse then $\leq$ is the usual natural order on $E(S)$, discussed in Section 2). We call an idempotent $f$ of $S$ \textit{primitive} if, for all $e\in E(S)$,
 \[ e\geq f \Rightarrow e=f. 
 \] 
 If all idempotents of $S$ are primitive, then we call $S$ a \textit{primitive semigroup}. 
  
 \begin{lemma} \label{non-per prim} Let $S$ be a HS with a non-periodic element contained in a maximal subgroup of $S$. Then $S$ is a primitive semigroup. 
 \end{lemma} 
 
 \begin{proof} Since $\langle e \rangle=\{e\}=\langle e \rangle_I$ for any $e\in E(S)$, the results of Lemma \ref{trans on Y} and Corollary \ref{iso max subgrps} easily extend to homogeneous semigroups, that is, Aut($S$) acts transitively on $E(S)$ and the maximal subgroups of $S$ are pairwise isomorphic. In particular, each maximal subgroup of $S$ is non-periodic. Let $e,f\in E(S)$ be such that $e\geq f$, and let $x\in H_f \cap \text{Inf}(S)$. Then
 \[ ex=e(fx)=(ef)x=fx=x=xf=x(fe)=(xf)e=xe
 \] 
 and so the map
 \[ \phi: \langle e,x \rangle \rightarrow \langle f,x \rangle 
 \]
 determined by $e\phi=f$ and $x\phi=x$ is an isomorphism. By the homogeneity of $S$  we may extend $\phi$ to an automorphism $\bar{\phi}$ of $S$. Since $\mathcal{H}$ is preserved under $\bar{\phi}$ we have $H_e\bar{\phi}=H_f$ and $H_f\bar{\phi}=H_x\bar{\phi}=H_x=H_f$. Hence $H_e=H_f$ and it follows that $e=f$ as required. 
 \end{proof} 

A regular semigroup $S$ which contains a primitive idempotent is called \textit{completely simple}, and we therefore obtain the following corollary to Lemma \ref{non-per prim}

\begin{corollary}\label{non per cr} Let $S$ be a regular homogeneous semigroup. If $S$ contains a non-periodic element in a subgroup of $S$ then $S$ is completely simple. 
\end{corollary} 

\begin{corollary} A non-periodic inverse HS is a group, homogeneous as a semigroup. 
\end{corollary} 

\begin{proof} Let $S$ be a non-periodic inverse HS. Then $S$ is a HIS, and is either bisimple or Clifford by Theorem \ref{cliff of bisimple}. 
If $S$ is bisimple, then there exists $x\in S$ such that $\langle x,x^{-1}\rangle$ is isomorphic to the bicyclic monoid with $xx^{-1}>x^{-1}x$. Since Aut($S$) acts transitively on Inf($S$), there exists an automorphism $\theta$ of $S$ with $x\theta=x^{-1}$. Hence $x^{-1}\theta=x$ as $S$ is inverse, and so $(xx^{-1})\theta=x^{-1}x$ and $(x^{-1}x)\theta=xx^{-1}$. This contradicts $\theta$ preserving the natural order on $E(S)$, and so $S$ is Clifford. Since $S$ is a union of groups, elements of infinite order are certainly contained in maximal subgroups, and thus $S$ is primitive by Lemma \ref{non-per prim}. However the semilattice $E(S)$ is primitive if and only if it is trivial, and so $S$ is a group. 
\end{proof} 

However the converse is not known, that is, is a homogeneous group an inverse HS? We give a positive answer for the class of abelian groups, thus completing our study into the homogeneity of commutative inverse semigroups:  

 \begin{proposition} A homogeneous abelian group is a HS. 
 \end{proposition} 
 
 \begin{proof} If $G$ is periodic then the result is clear, and so we assume $G$ is a non-periodic abelian homogeneous group with identity 1. Let $\phi:A \rightarrow B$ be an isomorphism between f.g. subsemigroups of $G$, and let $A_G,B_G$ denote the finitely generated subgroups of $G$ generated by $A$ and $B$, respectively. Since $G$ is abelian, each element of $A_G$ is of the form $uv^{-1}$ for some $u,v\in A$ and so we may take the map $\hat{\phi}:A_G\rightarrow B_G$ given by  
 \[ (uv^{-1})\hat{\phi}=(u\phi)(v\phi)^{-1}.
 \] 
 Then $\hat{\phi}$ is well defined and injective since, for any $uv^{-1},st^{-1}\in A_G$ we have 
\begin{align*} (uv^{-1})\hat{\phi}=(st^{-1})\hat{\phi} & \Leftrightarrow (u\phi)(v\phi)^{-1}=(s\phi) (t\phi)^{-1}  \\
&  \Leftrightarrow u\phi t\phi = s\phi v\phi \\
& \Leftrightarrow ut = sv \\
& \Leftrightarrow uv^{-1}=st^{-1}. 
 \end{align*} 
 If $ab^{-1}\in B_G$, then there exists $u,v\in A_G$ such that $u\phi=a$ and $v\phi=b$ since $\phi$ is surjective. Hence $(uv^{-1})\hat{\phi}=ab^{-1}$ and $\hat{\phi}$ is surjective. Finally, 
 \[ (uv^{-1})\hat{\phi}(st^{-1})\hat{\phi} = (u\phi)(v\phi)^{-1}(s\phi) (t\phi)^{-1} = (us)\phi ((vt)\phi)^{-1}=(us(vt)^{-1})\hat{\phi} = ((uv^{-1})(st^{-1}))\hat{\phi}
 \] 
 and $1\hat{\phi}=(uu^{-1})\hat{\phi}=(u\phi)(u\phi)^{-1}=1$ for any $u\in A$. It follows that $\hat{\phi}$ is an isomorphism, and extends $\phi$ since for all $u\in A$, 
 \[ u\hat{\phi}=u\phi (1\phi)^{-1}= u\phi. 
 \]
 Since any automorphism of $G$ which extends $\hat{\phi}$ additionally extends $\phi$, we have that $G$ is a HS by the homogeneity of $G$. 
 \end{proof}
 
From the proposition above and Theorem \ref{Periodic com} we obtain a complete classification of all commutative inverse HS, as either a periodic commutative HIS or a homogeneous non-periodic abelian group.    
 
\begin{open} Is a non-periodic homogeneous group a HS?
\end{open} 

We have an number of open problems, and we hope to answer these in due course.

\end{document}